\documentclass[twoside,reqno]{amsart}

\usepackage{latexsym}

\newcommand{\qdn}{\hspace*{-1.5mm}}
\newcommand{\qqdn}{\hspace*{-2.5mm}}
\newcommand{\xqdn}{\hspace*{-5.0mm}}
\newcommand{\xxqdn}{\hspace*{-10mm}}

\newcommand{\sst}{\scriptstyle}

\newcommand{\ffnk}[4]{\left[\qdn\ba{#1}#3\\#4\ea{\!\Big|\:#2}\right]}

\newcommand{\binm}{\binom}

\newcommand{\nnm}{\nonumber}
\newcommand{\be}{\begin{equation}}
\newcommand{\ee}{\end{equation}}
\newcommand{\ba}{\begin{array}}
\newcommand{\ea}{\end{array}}
\newcommand{\bmn}{\begin{eqnarray}}
\newcommand{\emn}{\end{eqnarray}}
\newcommand{\bnm}{\begin{eqnarray*}}
\newcommand{\enm}{\end{eqnarray*}}
\newcommand{\bln}{\begin{subequations}}
\newcommand{\eln}{\end{subequations}}

\newtheorem{thm}{Theorem}
\newtheorem{lemm}[thm]{Lemma}
\newtheorem{corl}[thm]{Corollary}

\newtheorem{entry}{Entry}

\newcommand{\bbtm}[4]{\bibitem{kn:#1}{#2,}~{#3,}~{#4.}}
\newcommand{\cito}[1]{\cite{kn:#1}}
\newcommand{\citu}[2]{\cite[#2]{kn:#1}}

\begin{document} 
{
\title{Whipple-type $_3F_2$-series
 and summation formulae involving generalized harmonic
numbers}
\author{$^{a}$Chuanan Wei, $^b$Xiaoxia Wang}


\dedicatory{
$^A$Department of Medical Informatics\\
  Hainan Medical University, Haikou 571199, China\\
$^B$Department of Mathematics\\
  Shanghai University, Shanghai 200444, China}

\thanks{ \emph{Email addresses}: weichuanan78@163.com (C. Wei),
 xwang913@126. com (X. Wang)}


\begin{abstract}
By means of the derivative operator and Whipple-type $_3F_2$-series
identities, two families of summation formulae involving generalized
harmonic numbers are established.
\end{abstract}

\maketitle\thispagestyle{empty}
\markboth{C. Wei}
         {Summation formulae involving generalized harmonic numbers}

\section{Introduction}
For a complex variable $x$, define the shifted factorial to be
\[(x)_{0}=1\quad \text{and}\quad (x)_{n}
=x(x+1)\cdots(x+n-1)\quad \text{with}\quad n\in\mathbb{N}.\]
 Following Andrews, Askey and Roy~\citu{andrews-r}{Chapter 2}, define the hypergeometric series by
\[\qdn_{r}F_s\ffnk{cccc}{z}{a_1,a_2,\cdots,a_r}{b_1,b_2,\cdots,b_s}
 \:=\:\sum_{k=0}^\infty\frac{(a_1)_k(a_2)_k\cdots(a_r)_k}{(b_1)_k(b_2)_k\cdots(b_s)_k}\frac{z^k}{k!},\]
where $\{a_{i}\}_{i\geq1}$ and $\{b_{j}\}_{j\geq1}$ are complex
parameters such that no zero factors appear in the denominators of
the summand on the right hand side. Then Whipple's $_3F_2$-series
identity (cf. \citu{andrews-r}{p. 149}) can be stated as
 \bmn\label{whipple}
 _3F_2\ffnk{cccc}{1}{a,1-a,b}{c,1+2b-c}
=\frac{\Gamma(\frac{c}{2})\Gamma(\frac{1+c}{2})\Gamma(b+\frac{1-c}{2})\Gamma(b+\frac{2-c}{2})}
{\Gamma(\frac{a+c}{2})\Gamma(\frac{1-a+c}{2})\Gamma(b+\frac{1+a-c}{2})\Gamma(b+\frac{2-a-c}{2})},
 \emn
where $Re(b)>0$ and $\Gamma(x)$ is the well-known gamma function
\[\xqdn\Gamma(x)=\int_{0}^{\infty}t^{x-1}e^{-t}dt\quad\text{with}\quad Re(x)>0.\]

For a complex number $x$ and a positive integer $\ell$, define
generalized harmonic numbers of $\ell$-order to be
\[H_{0}^{\langle \ell\rangle}(x)=0
\quad\text{and}\quad
 H_{n}^{\langle\ell\rangle}(x)=\sum_{k=1}^n\frac{1}{(x+k)^{\ell}}
 \quad\text{with}\quad n\in\mathbb{N}.\]
When $x=0$, they become harmonic numbers of $\ell$-order
\[H_{0}^{\langle \ell\rangle}=0
\quad\text{and}\quad
  H_{n}^{\langle \ell\rangle}
  =\sum_{k=1}^n\frac{1}{k^{\ell}} \quad\text{with}\quad n\in\mathbb{N}.\]
  Fixing $\ell=1$ in $H_{0}^{\langle \ell\rangle}(x)$ and $H_{n}^{\langle \ell\rangle}(x)$, we obtain
generalized harmonic numbers
\[H_{0}(x)=0
\quad\text{and}\quad H_{n}(x)
  =\sum_{k=1}^n\frac{1}{x+k} \quad\text{with}\quad n\in\mathbb{N}.\]
When $x=0$, they reduce to classical harmonic numbers
\[H_{0}=0\quad\text{and}\quad
H_{n}=\sum_{k=1}^n\frac{1}{k} \quad\text{with}\quad
n\in\mathbb{N}.\]

For a differentiable function $f(x)$, define the derivative operator
$\mathcal{D}_x^i$ by
 \bnm
\mathcal{D}_x^if(x)=\frac{d^i}{dx^i}f(x).
 \enm
When $i=1$, the corresponding sign can be simplified to
$\mathcal{D}_x$.

In order to explain the relation of the derivative operator and
generalized harmonic numbers, we introduce the following lemma.

\begin{lemm} \label{lemm-a}
 Let $x$ and  $\{a_j,b_j,c_j,d_j\}_{j=1}^s$ be all complex
numbers. Then
 \bnm
\mathcal{D}_x\prod_{j=1}^s\frac{a_jx+b_j}{c_jx+d_j}=\prod_{j=1}^s\frac{a_jx+b_j}{c_jx+d_j}
 \sum_{j=1}^s\frac{a_jd_j-b_jc_j}{(a_jx+b_j)(c_jx+d_j)}.
 \enm
\end{lemm}

\begin{proof}
It is not difficult to verify the case $s=1$ of Lemma \ref{lemm-a}.
Suppose that
 \bnm
\mathcal{D}_x\prod_{j=1}^m\frac{a_jx+b_j}{c_jx+d_j}=\prod_{j=1}^m\frac{a_jx+b_j}{c_jx+d_j}
 \sum_{j=1}^m\frac{a_jd_j-b_jc_j}{(a_jx+b_j)(c_jx+d_j)}
 \enm
is true. We can proceed as follows:
 \bnm
&&\xqdn\mathcal{D}_x\prod_{j=1}^{m+1}\frac{a_jx+b_j}{c_jx+d_j}=
 \mathcal{D}_x\bigg\{\prod_{j=1}^{m}\frac{a_jx+b_j}{c_jx+d_j}\frac{a_{m+1}x+b_{m+1}}{c_{m+1}x+d_{m+1}}\bigg\}\\
&&\xqdn\:=\:\frac{a_{m+1}x+b_{m+1}}{c_{m+1}x+d_{m+1}}\mathcal{D}_x\prod_{j=1}^{m}\frac{a_jx+b_j}{c_jx+d_j}
+\prod_{j=1}^{m}\frac{a_jx+b_j}{c_jx+d_j}\mathcal{D}_x\frac{a_{m+1}x+b_{m+1}}{c_{m+1}x+d_{m+1}}\\
&&\xqdn\:=\:\frac{a_{m+1}x+b_{m+1}}{c_{m+1}x+d_{m+1}}\prod_{j=1}^m\frac{a_jx+b_j}{c_jx+d_j}\sum_{j=1}^m\frac{a_jd_j-b_jc_j}{(a_jx+b_j)(c_jx+d_j)}\\
&&\xqdn\:+\:\prod_{j=1}^{m}\frac{a_jx+b_j}{c_jx+d_j}\frac{a_{m+1}d_{m+1}-b_{m+1}c_{m+1}}{(c_{m+1}x+d_{m+1})^2}\\
&&\xqdn\:=\:\prod_{j=1}^{m+1}\frac{a_jx+b_j}{c_jx+d_j}
 \bigg\{\sum_{j=1}^m\frac{a_jd_j-b_jc_j}{(a_jx+b_j)(c_jx+d_j)}+\frac{a_{m+1}d_{m+1}-b_{m+1}c_{m+1}}{(a_{m+1}x+b_{m+1})(c_{m+1}x+d_{m+1})}\bigg\}\\
&&\xqdn\:=\:\prod_{j=1}^{m+1}\frac{a_jx+b_j}{c_jx+d_j}
 \sum_{j=1}^{m+1}\frac{a_jd_j-b_jc_j}{(a_jx+b_j)(c_jx+d_j)}.
 \enm
This proves Lemma \ref{lemm-a} inductively.
\end{proof}

Setting $a_j=1,b_j=r-j+1,c_j=0$ and $d_j=j$ in Lemma \ref{lemm-a},
it is easy to see that
$$\mathcal{D}_x\:\binm{x+r}{s}=\binm{x+r}{s}\big\{H_r(x)-H_{r-s}(x)\big\},$$
where $r,s\in\mathbb{N}_0$ with $s\leq r$. Besides, we have the
following relation:
$$\mathcal{D}_xH_{n}^{\langle \ell\rangle}(x)=-\ell H_{n}^{\langle
\ell+1\rangle}(x).$$

As pointed out by Richard Askey (cf. \cito{andrews}), expressing
harmonic numbers in accordance with differentiation of binomial
coefficients can be traced back to Issac Newton.
 In 2003, Paule and Schneider
\cito{paule} computed the family of series:
 \bnm
\quad
W_n(\alpha)=\sum_{k=0}^n\binm{n}{k}^{\alpha}\{1+\alpha(n-2k)H_k\}
 \enm
with $\alpha=1,2,3,4,5$ by combining this way with Zeilberger's
algorithm for definite hypergeometric sums. In terms of the
derivative operator and the hypergeometric form of Andrews'
$q$-series transformation, Krattenthaler and Rivoal
\cito{krattenthaler} deduced general Paule-Schneider type identities
with $\alpha$ being a positive integer.  More results from
differentiation of binomial coefficients can be seen in the papers
\cite{kn:sofo-a,kn:wang-b,kn:wei-a,kn:wei-c,kn:wei-d}. For different
ways and related results, the reader may refer to
\cite{kn:chen,kn:kronenburg-a,
kn:kronenburg-b,kn:schneider,kn:sofo-b,kn:wang-a}. It should be
mentioned that Sun \cite{kn:sun-a,kn:sun-b} showed recently some
congruence relations concerning harmonic numbers to us.

Inspired by the work just mentioned, we shall explore, according to
the derivative operator and Whipple-type $_3F_2$-series identities,
closed expressions for the following two families of series
involving generalized harmonic numbers:
 \bnm
 &&\qdn\xqdn\sum_{k=0}^{n}(-1)^k\binm{n}{k}\frac{\binm{n+k}{k}}{\binm{x+k}{k}}k^tH_{k}^{\langle2\rangle}(x),\\
 &&\qdn\xqdn\sum_{k=0}^{n}(-1)^k\binm{n}{k}\frac{\binm{n+k}{k}}{\binm{x+k}{k}}k^tH_{k}^2(x),
 \enm
where $t\in\mathbb{N}_0$. When $x$ is a nonnegative integer $p$,
they give closed expressions for the following two classes of series
on harmonic numbers:
 \bnm
 &&\sum_{k=0}^{n}(-1)^k\binm{n+k}{k}\binm{p+n}{n-k}k^tH_{p+k}^{\langle2\rangle},\\
 &&\sum_{k=0}^{n}(-1)^k\binm{n+k}{k}\binm{p+n}{n-k}k^tH_{p+k}^{2}.
 \enm
Due to limit of space, our explicit formulae are offered only for
$t=0,1,2$ in this paper.
\section{The first family of summation formulae involving\\ generalized harmonic numbers}
\begin{lemm}\label{lemm-b}
Let $a$, $b$ and $c$ be all complex numbers. Then
 \bnm
 _3F_2\ffnk{cccc}{1}{a,1-a,1+b}{1+c,1+2b-c}
&&\xqdn\!=\frac{1}{b}\frac{\Gamma(\frac{1+c}{2})\Gamma(\frac{2+c}{2})\Gamma(b+\frac{1-c}{2})\Gamma(b+\frac{2-c}{2})}
{\Gamma(\frac{a+c}{2})\Gamma(\frac{1-a+c}{2})\Gamma(b+\frac{1+a-c}{2})\Gamma(b+\frac{2-a-c}{2})}\\
&&\xqdn\!+\:\frac{1}{b}\frac{\Gamma(\frac{1+c}{2})\Gamma(\frac{2+c}{2})\Gamma(b+\frac{1-c}{2})\Gamma(b+\frac{2-c}{2})}
{\Gamma(\frac{1+a+c}{2})\Gamma(\frac{2-a+c}{2})\Gamma(b+\frac{a-c}{2})\Gamma(b+\frac{1-a-c}{2})}
 \enm
provided that $Re(b)>0$.
\end{lemm}

\begin{proof}
Perform the replacement $c\to 1+c$ in \eqref{whipple} to get
 \bmn\label{whipple-a}
\quad
_3F_2\ffnk{cccc}{1}{a,1-a,b}{1+c,2b-c}\frac{\Gamma(\frac{1+c}{2})\Gamma(\frac{2+c}{2})\Gamma(b-\frac{c}{2})\Gamma(b-\frac{c-1}{2})}
{\Gamma(\frac{1+a+c}{2})\Gamma(\frac{2-a+c}{2})\Gamma(b+\frac{a-c}{2})\Gamma(b+\frac{1-a-c}{2})}.
 \emn
It is routine to show the continuous relation
 \bnm\quad
 _3F_2\ffnk{cccc}{1}{a,1-a,1+b}{1+c,1+2b-c}
  =\frac{c}{2b}
 {_3F_2}\ffnk{cccc}{1}{a,1-a,b}{c,1+2b-c}
+\frac{2b-c}{2b}
  {_3F_2}\ffnk{cccc}{1}{a,1-a,b}{1+c,2b-c}.
 \enm
Calculating, respectively, the two series on the right hand side by
\eqref{whipple} and \eqref{whipple-a}, we gain Lemma \ref{lemm-b}.
\end{proof}

\begin{thm} \label{thm-a}
Let $x$ be a complex number. Then
 \bnm
&&\xxqdn\sum_{k=0}^n(-1)^k\binm{n}{k}\frac{\binm{n+k}{k}}{\binm{x+k}{k}}H_{k}^{\langle2\rangle}(x)
=\frac{(-1)^n}{2}\frac{\binm{-x+n}{n}}{\binm{x+n}{n}}\bigg\{\Big[H_{n}^{\langle2\rangle}(x)-H_{n}^{\langle2\rangle}(-x)\Big]+\tfrac{4n}{x(x-n)^2}\\
&&\xxqdn\:\:+\:\Big[H_n(x)-H_{n}(-x)-H_n(\tfrac{x-n}{2})\Big]\Big[H_n(x)-H_{n}(-x)-H_n(\tfrac{x-n}{2})-\tfrac{2(x+n)}{x(x-n)}\Big]\bigg\}.
 \enm
\end{thm}

\begin{proof}
The case $a=-n$, $b=x$ and $c=y$ of Lemma \ref{lemm-b} reads as
 \bmn\label{whipple-b}
\sum_{k=0}^n(-1)^k\binm{n}{k}\frac{\binm{n+k}{k}\binm{x+k}{k}}{\binm{y+k}{k}\binm{2x-y+k}{k}}
&&\xqdn\!=\frac{2x-y}{2x}\frac{\binm{\frac{y-n-1}{2}+n}{n}\binm{y-2x+n}{n}}{\binm{\frac{y-n-1}{2}-x+n}{n}\binm{y+n}{n}}
\nnm\\
&&\xqdn\!+\:\frac{(2x-y)(y-n)}{2x(2x-y+n)}\frac{\binm{\frac{y-n}{2}+n}{n}\binm{y-2x+n}{n}}{\binm{\frac{y-n}{2}-x+n}{n}\binm{y+n}{n}}.
 \emn
 Applying the derivative operator $\mathcal{D}_y$ to both sides of
 it, we achieve
 \bnm
\quad\sum_{k=0}^n(-1)^k\binm{n}{k}\frac{\binm{n+k}{k}\binm{x+k}{k}}{\binm{y+k}{k}\binm{2x-y+k}{k}}
\big\{H_k(2x-y)-H_k(y)\big\}=\Omega_n(x,y),
 \enm
where the symbol on the right hand side stands for
 \bnm
\Omega_n(x,y)&&\xqdn\!=\frac{2x-y}{2x}\frac{\binm{\frac{y-n-1}{2}+n}{n}\binm{y-2x+n}{n}}{\binm{\frac{y-n-1}{2}-x+n}{n}\binm{y+n}{n}}\\
&&\xqdn\!\times\:\Big\{\tfrac{1}{2}H_n(\tfrac{y-n-1}{2})-\tfrac{1}{2}H_n(\tfrac{y-n-1}{2}-x)+H_{n+1}(y-2x-1)-H_n(y)\Big\}\\
&&\xqdn\!+\:\frac{(2x-y)(y-n)}{2x(2x-y+n)}\frac{\binm{\frac{y-n}{2}+n}{n}\binm{y-2x+n}{n}}{\binm{\frac{y-n}{2}-x+n}{n}\binm{y+n}{n}}\\
&&\xqdn\!\times\:\Big\{\tfrac{1}{2}H_{n+1}(\tfrac{y-n-2}{2})-\tfrac{1}{2}H_{n+1}(\tfrac{y-n-2}{2}-x)+H_{n+1}(y-2x-1)-H_n(y)\Big\}.
 \enm
 The last equation can be reformulated as
 \bnm
\qquad\sum_{k=0}^n(-1)^k\binm{n}{k}\frac{\binm{n+k}{k}\binm{x+k}{k}}{\binm{y+k}{k}\binm{2x-y+k}{k}}
\sum_{i=1}^k\frac{1}{(2x-y+i)(y+i)}=\frac{\Omega_n(x,y)}{2(y-x)}.
 \enm
Finding the limit $y\to x$ of it by using the relation
 \bnm
&&\text{Lim}_{y\to x}\frac{\Omega_n(x,y)}{2(y-x)}\\
&&\:=\:\text{Lim}_{y\to x}
\frac{\mathcal{D}_y\Omega_n(x,y)}{2}\\
&&\:=\:\frac{(-1)^n}{2}\frac{\binm{-x+n}{n}}{\binm{x+n}{n}}
\bigg\{\Big[H_{n}^{\langle2\rangle}(x)-H_{n}^{\langle2\rangle}(-x)\Big]+\tfrac{4n}{x(x-n)^2}\\
&&\:+\:\Big[H_n(x)-H_{n}(-x)-H_n(\tfrac{x-n}{2})\Big]
\Big[H_n(x)-H_{n}(-x)-H_n(\tfrac{x-n}{2})-\tfrac{2(x+n)}{x(x-n)}\Big]\bigg\}
 \enm
from L'H\^{o}spital rule, we attain Theorem \ref{thm-a}.
\end{proof}

\begin{corl}[Harmonic number identity]\label{corl-a}
\bnm
 \quad\sum_{k=0}^n(-1)^k\binm{n}{k}\binm{n+k}{k}H_{k}^{\langle2\rangle}=
 \begin{cases}
 (-1)^n\Big\{2H_{n}^{\langle2\rangle}-H_{\frac{n}{2}}^{\langle2\rangle}\Big\},&n=0\,(\qqdn\mod2);\\[2mm]
  (-1)^n\Big\{2H_{n}^{\langle2\rangle}-H_{\frac{n-1}{2}}^{\langle2\rangle}\Big\},&n=1\,(\qqdn\mod2).
\end{cases}
 \enm
\end{corl}

\begin{proof}
When $n=0\,(\qqdn\mod2)$, Theorem \ref{thm-a} can be manipulated as
\bnm
&&\xqdn\sum_{k=0}^n(-1)^k\binm{n}{k}\frac{\binm{n+k}{k}}{\binm{x+k}{k}}H_{k}^{\langle2\rangle}(x)
=\frac{(-1)^n}{2}\frac{\binm{-x+n}{n}}{\binm{x+n}{n}}\\
&&\xqdn\:\times\:\bigg\{\Big[H_{n}^{\langle2\rangle}(x)-H_{n}^{\langle2\rangle}(-x)\Big]
+\Big[H_n(x)-H_{n}(-x)-H_{\frac{n}{2}}(\tfrac{x}{2})+H_{\frac{n}{2}}(-\tfrac{x}{2})-\tfrac{2}{n-x}\Big]\\
&&\xqdn\:\times\:\Big[H_n(x)-H_{n}(-x)-H_{\frac{n}{2}}(\tfrac{x}{2})+H_{\frac{n}{2}}(-\tfrac{x}{2})+\tfrac{2}{n-x}\Big]+\tfrac{4}{(n-x)^2}\\
&&\xqdn\:-\:\frac{2[H_n(x)-H_{n}(-x)-H_{\frac{n}{2}}(\tfrac{x}{2})+H_{\frac{n}{2}}(-\tfrac{x}{2})]}{x}\bigg\}.
 \enm
Taking the limit $x\to0$ of it by utilizing the relation
 \bnm
&&\xxqdn\text{Lim}_{x\to0}\frac{H_n(x)-H_{n}(-x)-H_{\frac{n}{2}}(\tfrac{x}{2})+H_{\frac{n}{2}}(-\tfrac{x}{2})}{x}\\
&&\xxqdn\:\:=\:\text{Lim}_{x\to0}\mathcal{D}_x\big\{H_n(x)-H_{n}(-x)-H_{\frac{n}{2}}(\tfrac{x}{2})+H_{\frac{n}{2}}(-\tfrac{x}{2})\big\}\\
&&\xxqdn\:\:=\:H_{\frac{n}{2}}^{\langle2\rangle}-2H_{n}^{\langle2\rangle}
 \enm
 from L'H\^{o}spital
rule, we obtain
 \bmn\label{relation-a}
\sum_{k=0}^n(-1)^k\binm{n}{k}\binm{n+k}{k}H_{k}^{\langle2\rangle}
=(-1)^n\Big\{2H_{n}^{\langle2\rangle}-H_{\frac{n}{2}}^{\langle2\rangle}\Big\}.
 \emn

When $n=1\,(\qqdn\mod2)$, Theorem \ref{thm-a} can be restated as
\bnm
&&\sum_{k=0}^n(-1)^k\binm{n}{k}\frac{\binm{n+k}{k}}{\binm{x+k}{k}}H_{k}^{\langle2\rangle}(x)\\
&&\:=\:\frac{(-1)^n}{2}\frac{\binm{-x+n}{n}}{\binm{x+n}{n}}\bigg\{\Big[H_{n}^{\langle2\rangle}(x)-H_{n}^{\langle2\rangle}(-x)\Big]
+\Big[H_n(x)-H_{n}(-x)-H_n(\tfrac{x-n}{2})\Big]^2\\
&&\:+\:\frac{2}{(n-x)^2}\frac{(n^2-x^2)[H_n(x)-H_{n}(-x)-H_n(\tfrac{x-n}{2})]+2n}{x}\bigg\}.
 \enm
Finding the limit $x\to0$ of it by using the relation
 \bnm
&&\text{Lim}_{x\to0}\frac{(n^2-x^2)[H_n(x)-H_{n}(-x)-H_n(\tfrac{x-n}{2})]+2n}{x}\\
&&\:\:=\:\text{Lim}_{x\to0}\mathcal{D}_x\big\{(n^2-x^2)[H_n(x)-H_{n}(-x)-H_n(\tfrac{x-n}{2})]+2n\big\}\\
&&\:\:=\:\frac{n^2}{2}\big[H_{n}^{\langle2\rangle}(-\tfrac{n}{2})-4H_{n}^{\langle2\rangle}\big]\\
&&\:\:=\:n^2\big[2H_{n}^{\langle2\rangle}-H_{\frac{n-1}{2}}^{\langle2\rangle}-\tfrac{2}{n^2}\big]
 \enm
 from L'H\^{o}spital
rule, we get
 \bmn\label{relation-b}
\sum_{k=0}^n(-1)^k\binm{n}{k}\binm{n+k}{k}H_{k}^{\langle2\rangle}
=(-1)^n\Big\{2H_{n}^{\langle2\rangle}-H_{\frac{n-1}{2}}^{\langle2\rangle}\Big\}.
 \emn
Then \eqref{relation-a} and \eqref{relation-b} are unified to
Corollary \ref{corl-a}.
\end{proof}

\begin{corl} \label{corl-b}
Let $p$ be a positive integer satisfying $0<p\leq n$. Then
 \bnm
 &&\sum_{k=0}^n(-1)^k\binm{n+k}{k}\binm{p+n}{n-k}H_{p+k}^{\langle2\rangle}=\frac{(-1)^{n-p+1}}{p\binm{n}{p}}\\
 &&\:\times\:\begin{cases}
 H_{n+p}- H_{n-p}-H_{\frac{n+p}{2}}+H_{\frac{n-p}{2}},&n-p=0\,(\qqdn\mod2);\\[2mm]
  H_{n+p}-H_{n-p}-H_{\frac{n+p-1}{2}}+H_{\frac{n-p-1}{2}},&n-p=1\,(\qqdn\mod2).
\end{cases}
 \enm
\end{corl}

\begin{proof}
When $n-p=0\,(\qqdn\mod2)$, Theorem \ref{thm-a} can be written as
\bnm
&&\xqdn\sum_{k=0}^n(-1)^k\binm{n}{k}\frac{\binm{n+k}{k}}{\binm{x+k}{k}}H_{k}^{\langle2\rangle}(x)
=\frac{(-1)^{n+p}}{2}\frac{x-p}{x}\frac{\binm{-x+n}{n-p}\binm{x}{p}}{\binm{x+n}{n}\binm{n}{p}}\\
&&\xqdn\:\times\:\bigg\{\Big[H_{n}^{\langle2\rangle}(x)-H_{p-1}^{\langle2\rangle}(x-p)-H_{n-p}^{\langle2\rangle}(p-x)\Big]+\frac{4n}{x(x-n)^2}\\
&&\xqdn\:\:+\:\Big[H_n(x)+H_{p-1}(x-p)-H_{n-p}(p-x)-H_{\frac{n+p}{2}}(\tfrac{x-p}{2})+H_{\frac{n-p-2}{2}}(\tfrac{p-x}{2})\Big]\\
&&\xqdn\:\:\times\:\Big[H_n(x)+H_{p-1}(x-p)-H_{n-p}(p-x)-H_{\frac{n+p}{2}}(\tfrac{x-p}{2})+H_{\frac{n-p-2}{2}}(\tfrac{p-x}{2})-\tfrac{2(x+n)}{x(x-n)}\Big]\\
&&\xqdn\:\:-\:\frac{2[H_n(x)+H_{p-1}(x-p)-H_{n-p}(p-x)-H_{\frac{n+p}{2}}(\tfrac{x-p}{2})+H_{\frac{n-p-2}{2}}(\tfrac{p-x}{2})-\tfrac{x+n}{x(x-n)}]}{x-p}\bigg\}.
 \enm
 Taking the limit $x\to p$ of it, we gain
\bnm
&&\sum_{k=0}^n(-1)^k\binm{n+k}{k}\binm{p+n}{n-k}H_{p+k}^{\langle2\rangle}\\
&&\:=\:H_{p}^{\langle2\rangle}\sum_{k=0}^n(-1)^k\binm{n+k}{k}\binm{p+n}{n-k}\\
&&\:+\:\frac{(-1)^{n-p+1}}{p\binm{n}{p}}\Big\{H_{n+p}-
H_{n-p}-H_{\frac{n+p}{2}}+H_{\frac{n-p}{2}}\Big\}.
 \enm
Evaluating the series on the right hand side by \eqref{whipple-b},
we have
 \bmn\label{relation-c}
&&\sum_{k=0}^n(-1)^k\binm{n+k}{k}\binm{p+n}{n-k}H_{p+k}^{\langle2\rangle}
 \nnm\\
&&\:=\:\frac{(-1)^{n-p+1}}{p\binm{n}{p}}\Big\{H_{n+p}-
H_{n-p}-H_{\frac{n+p}{2}}+H_{\frac{n-p}{2}}\Big\}.
 \emn

When $n-p=1\,(\qqdn\mod2)$, Theorem \ref{thm-a} can be reformulated
as
 \bnm
&&\xqdn\sum_{k=0}^n(-1)^k\binm{n}{k}\frac{\binm{n+k}{k}}{\binm{x+k}{k}}H_{k}^{\langle2\rangle}(x)
=\frac{(-1)^{n+p}}{2}\frac{x-p}{x}\frac{\binm{-x+n}{n-p}\binm{x}{p}}{\binm{x+n}{n}\binm{n}{p}}\\
&&\xqdn\:\times\:\bigg\{\Big[H_{n}^{\langle2\rangle}(x)-H_{p-1}^{\langle2\rangle}(x-p)-H_{n-p}^{\langle2\rangle}(p-x)\Big]+\frac{4n}{x(x-n)^2}\\
&&\xqdn\:\:+\:\Big[H_n(x)+H_{p-1}(x-p)-H_{n-p}(p-x)-H_{n}(\tfrac{x-n}{2})\Big]\\
&&\xqdn\:\:\times\:\Big[H_n(x)+H_{p-1}(x-p)-H_{n-p}(p-x)-H_{n}(\tfrac{x-n}{2})-\tfrac{2(x+n)}{x(x-n)}\Big]\\
&&\xqdn\:\:+\:\frac{2[H_n(x)+H_{p-1}(x-p)-H_{n-p}(p-x)-H_{n}(\tfrac{x-n}{2})-\tfrac{x+n}{x(x-n)}]}{x-p}\bigg\}.
 \enm
Finding the limit $x\to p$ of it, we achieve
 \bnm
&&\quad\:\:\sum_{k=0}^n(-1)^k\binm{n+k}{k}\binm{p+n}{n-k}H_{p+k}^{\langle2\rangle}\\
&&\quad\:\:\:=\:H_{p}^{\langle2\rangle}\sum_{k=0}^n(-1)^k\binm{n+k}{k}\binm{p+n}{n-k}\\
&&\quad\:\:\:+\:\frac{(-1)^{n-p+1}}{p\binm{n}{p}}\Big\{H_{n+p}-
H_{n-p}-H_{\frac{n+p-1}{2}}+H_{\frac{n-p-1}{2}}\Big\}.
 \enm
Calculating the series on the right hand side by \eqref{whipple-b},
we attain
 \bmn\label{relation-d}
&&\xqdn\sum_{k=0}^n(-1)^k\binm{n+k}{k}\binm{p+n}{n-k}H_{p+k}^{\langle2\rangle}
 \nnm\\
&&\xqdn\:=\:\frac{(-1)^{n-p+1}}{p\binm{n}{p}}\Big\{H_{n+p}-
H_{n-p}-H_{\frac{n+p-1}{2}}+H_{\frac{n-p-1}{2}}\Big\}.
 \emn
Then \eqref{relation-c} and \eqref{relation-d} are unified to
Corollary \ref{corl-b}.
\end{proof}

\begin{corl}[$p=n$ in Corollary \ref{corl-b}] \label{corl-c}
 \bnm
 \xqdn\sum_{k=0}^n(-1)^k\binm{n}{k}H_{n+k}^{\langle2\rangle}=\frac{1}{n\binm{2n}{n}}\Big\{H_{n}-H_{2n}\Big\}.
 \enm
\end{corl}

\begin{corl} \label{corl-d}
Let $p$ be a positive integer with $p>n$. Then
 \bnm
 \quad\sum_{k=0}^n(-1)^k\binm{n+k}{k}\binm{p+n}{n-k}H_{p+k}^{\langle2\rangle}=\frac{1}{2}\binm{p-1}{n}
 \Big\{H_{p+n}^{\langle2\rangle}+H_{p-n}^{\langle2\rangle}+A(p,n)\Big\},
 \enm
where the expression on the right hand side is
 \bnm
 A(p,n)=\begin{cases}
(H_{p+n}- H_{p-n}-H_{\frac{p+n}{2}}+H_{\frac{p-n}{2}})\\
\times(H_{p+n}- H_{p-n}-H_{\frac{p+n}{2}}+H_{\frac{p-n-2}{2}}),&p-n=0\,(\qqdn\mod2);\\[2mm]
  (H_{p+n}- H_{p-n}-H_{\frac{p+n-1}{2}}+H_{\frac{p-n-1}{2}})\\
\times(H_{p+n}-H_{p-n}-H_{\frac{p+n-1}{2}}+H_{\frac{p-n-1}{2}}+\frac{2}{p-n}),&p-n=1\,(\qqdn\mod2).
\end{cases}
 \enm
\end{corl}

\begin{proof}
When $p-n=0\,(\qqdn\mod2)$, the case $x=p$ of Theorem \ref{thm-a}
can be manipulated as
  \bnm
&&\xxqdn\sum_{k=0}^n(-1)^k\binm{n+k}{k}\binm{p+n}{n-k}H_{p+k}^{\langle2\rangle}
=H_{p}^{\langle2\rangle}\sum_{k=0}^n(-1)^k\binm{n+k}{k}\binm{p+n}{n-k}\\
&&\xxqdn\:\:+\:\frac{1}{2}\binm{p-1}{n}
 \Big\{H_{p+n}^{\langle2\rangle}-2H_{p}^{\langle2\rangle}+H_{p-n}^{\langle2\rangle}+(H_{p+n}-
 H_{p-n}-H_{\frac{p+n}{2}}+H_{\frac{p-n}{2}})\\
&&\quad\qquad\:\:\times\:(H_{p+n}-
H_{p-n}-H_{\frac{p+n}{2}}+H_{\frac{p-n-2}{2}})\Big\}.
 \enm
Evaluating the series on the right hand side by \eqref{whipple-b},
we obtain
 \bmn\label{relation-e}
&&\xqdn\sum_{k=0}^n(-1)^k\binm{n+k}{k}\binm{p+n}{n-k}H_{p+k}^{\langle2\rangle}=\frac{1}{2}\binm{p-1}{n}
 \nnm\\\nnm
&&\xqdn\:\:\times\:
 \Big\{H_{p+n}^{\langle2\rangle}+H_{p-n}^{\langle2\rangle}+(H_{p+n}-
 H_{p-n}-H_{\frac{p+n}{2}}+H_{\frac{p-n}{2}})\\
&&\xqdn\:\:\times\:(H_{p+n}-
H_{p-n}-H_{\frac{p+n}{2}}+H_{\frac{p-n-2}{2}})\Big\}.
 \emn

When $p-n=1\,(\qqdn\mod2)$, the case $x=p$ of Theorem \ref{thm-a}
can be restated as
  \bnm
&&\xqdn\sum_{k=0}^n(-1)^k\binm{n+k}{k}\binm{p+n}{n-k}H_{p+k}^{\langle2\rangle}
=H_{p}^{\langle2\rangle}\sum_{k=0}^n(-1)^k\binm{n+k}{k}\binm{p+n}{n-k}\\
&&\xqdn\:\:+\:\frac{1}{2}\binm{p-1}{n}
 \Big\{H_{p+n}^{\langle2\rangle}-2H_{p}^{\langle2\rangle}+H_{p-n}^{\langle2\rangle}+(H_{p+n}- H_{p-n}-H_{\frac{p+n-1}{2}}+H_{\frac{p-n-1}{2}})\\
&&\qquad\qquad\:\:\times\:(H_{p+n}-H_{p-n}-H_{\frac{p+n-1}{2}}+H_{\frac{p-n-1}{2}}+\tfrac{2}{p-n})\Big\}.
 \enm
Calculating the series on the right hand side by \eqref{whipple-b},
we get
 \bmn\label{relation-f}
&&\xqdn\sum_{k=0}^n(-1)^k\binm{n+k}{k}\binm{p+n}{n-k}H_{p+k}^{\langle2\rangle}=\frac{1}{2}\binm{p-1}{n}
 \nnm\\\nnm
&&\xqdn\:\:\times\:
 \Big\{H_{p+n}^{\langle2\rangle}+H_{p-n}^{\langle2\rangle}+(H_{p+n}- H_{p-n}-H_{\frac{p+n-1}{2}}+H_{\frac{p-n-1}{2}})\\
&&\xqdn\:\:\times\:(H_{p+n}-H_{p-n}-H_{\frac{p+n-1}{2}}+H_{\frac{p-n-1}{2}}+\tfrac{2}{p-n})\Big\}.
 \emn
Then \eqref{relation-e} and \eqref{relation-f} are unified to
Corollary \ref{corl-d}.
\end{proof}

\begin{lemm}\label{lemm-c}
Let $a$, $b$ and $c$ be all complex numbers. Then
 \bnm
&&\xxqdn\sum_{k=0}^{\infty}k\frac{(a)_k(1-a)_k(1+b)_k}{k!(1+c)_k(1+2b-c)_k}\\
&&\xxqdn\:=\:\frac{a^2-a+bc-c^2}{b(1-b)}\frac{\Gamma(\frac{1+c}{2})\Gamma(\frac{2+c}{2})\Gamma(b+\frac{1-c}{2})\Gamma(b+\frac{2-c}{2})}
{\Gamma(\frac{1+a+c}{2})\Gamma(\frac{2-a+c}{2})\Gamma(b+\frac{a-c}{2})\Gamma(b+\frac{1-a-c}{2})}\\
&&\xxqdn\:\:+\:\frac{a^2-a-2b^2+3bc-c^2}{b(1-b)}\frac{\Gamma(\frac{1+c}{2})\Gamma(\frac{2+c}{2})\Gamma(b+\frac{1-c}{2})\Gamma(b+\frac{2-c}{2})}
{\Gamma(\frac{a+c}{2})\Gamma(\frac{1-a+c}{2})\Gamma(b+\frac{1+a-c}{2})\Gamma(b+\frac{2-a-c}{2})}
 \enm
provided that $Re(b-1)>0$.
\end{lemm}

\begin{proof}
It is not difficult to verify the continuous relation
 \bnm\qquad
 _3F_2\ffnk{cccc}{1}{a,-a,b}{c,1+2b-c}=\frac{1}{2}
 {_3F_2}\ffnk{cccc}{1}{a,1-a,b}{c,1+2b-c}
 +\frac{1}{2}{_3F_2}\ffnk{cccc}{1}{1+a,-a,b}{c,1+2b-c}.
 \enm
 Evaluating the series on the right hand side by
 \eqref{whipple}, we gain
 \bmn\label{whipple-c}
 _3F_2\ffnk{cccc}{1}{a,-a,b}{c,1+2b-c}&&\xqdn\!=\frac{1}{2}
 \frac{\Gamma(\frac{c}{2})\Gamma(\frac{1+c}{2})\Gamma(b+\frac{1-c}{2})\Gamma(b+\frac{2-c}{2})}
{\Gamma(\frac{a+c}{2})\Gamma(\frac{1-a+c}{2})\Gamma(b+\frac{1+a-c}{2})\Gamma(b+\frac{2-a-c}{2})}
\nnm\\&&\xqdn\!+\:\frac{1}{2}\frac{\Gamma(\frac{c}{2})\Gamma(\frac{1+c}{2})\Gamma(b+\frac{1-c}{2})\Gamma(b+\frac{2-c}{2})}
{\Gamma(\frac{1+a+c}{2})\Gamma(\frac{c-a}{2})\Gamma(b+\frac{2+a-c}{2})\Gamma(b+\frac{1-a-c}{2})}.
 \emn
 By means of Kummer's transformation formula (cf.
\citu{andrews-r}{p. 142}):
  \bmn \label{kummer} \quad
_3F_2\ffnk{cccc}{1}{a,b,c}{d,e}
 =\frac{\Gamma(e)\Gamma(d+e-a-b-c)}{\Gamma(e-a)\Gamma(d+e-b-c)}
{_3F_2}\ffnk{cccc}{1}{a,d-b,d-c}{d,d+e-b-c},
 \emn
we have
  \bnm
  _3F_2\ffnk{cccc}{1}{a,b,c}{1+a-b,a-c}
 &&\xqdn\!=\frac{\Gamma(1+a-b)\Gamma(1+a-2b-2c)}{\Gamma(1+a-b-c)\Gamma(1+a-2b-c)}\\
 &&\xqdn\!\times\:{_3F_2}\ffnk{cccc}{1}{c,-c,a-b-c}{a-c,1+a-2b-c}.
 \enm
Calculating the series on the right hand side by \eqref{whipple-c},
we achieve
 \bmn\label{dixon-a}
 &&\xxqdn\qqdn_3F_2\ffnk{cccc}{1}{a,b,c}{1+a-b,a-c}
  \nnm\\\nnm
&&\xxqdn\qqdn\:\:=\:\frac{1}{2^{1+c}}\frac{\Gamma(1+a-b)\Gamma(\frac{1+a}{2}-b-c)\Gamma(\frac{a-c}{2})\Gamma(\frac{1+a-c}{2})}
{\Gamma(1+a-b-c)\Gamma(\frac{a}{2})\Gamma(\frac{1+a}{2}-b)\Gamma(\frac{1+a}{2}-c)}
\\
&&\xxqdn\qqdn\:\:+\:\,\frac{1}{2^{1+c}}\frac{\Gamma(1+a-b)\Gamma(\frac{2+a}{2}-b-c)\Gamma(\frac{a-c}{2})\Gamma(\frac{1+a-c}{2})}
{\Gamma(1+a-b-c)\Gamma(\frac{1+a}{2})\Gamma(\frac{2+a}{2}-b)\Gamma(\frac{a}{2}-c)}.
 \emn
It is easy to see the relation
 \bnm
 _4F_3\ffnk{cccc}{1}{a,b,c,1+x}{2+a-b,1+a-c,x}&&\xqdn\!=\frac{a(x+b-a-1)}{x(b-1)}
 {_3F_2}\ffnk{cccc}{1}{1+a,b,c}{2+a-b,1+a-c}\\
 &&\xqdn\!+\:\frac{(1+a-b)(a-x)}{x(b-1)}{_3F_2}\ffnk{cccc}{1}{a,b,c}{1+a-b,1+a-c}.
 \enm
Evaluating, respectively, the two series on the right hand side by
\eqref{dixon-a} and Dixon's $_3F_2$-series identity(cf.
\citu{andrews-r}{p. 72}):
 \bmn\label{dixon}
 &&\xxqdn\qqdn_3F_2\ffnk{cccc}{1}{a,b,c}{1+a-b,1+a-c}
  \nnm\\
&&\xxqdn\qqdn\:\:=\:\frac{\Gamma(\frac{2+a}{2})\Gamma(1+a-b)\Gamma(1+a-c)\Gamma(\frac{2+a}{2}-b-c)}
{\Gamma(1+a)\Gamma(\frac{2+a}{2}-b)\Gamma(\frac{2+a}{2}-c)\Gamma(1+a-b-c)},
 \emn
we attain
 \bnm
 &&_4F_3\ffnk{cccc}{1}{a,b,c,1+x}{2+a-b,1+a-c,x}\\
&&\:=\:\frac{a(1+a-b-2c)-x(2+a-2b-2c)}{2x(b-1)}\\
&&\:\:\times\:\frac{\Gamma(\frac{2+a}{2})\Gamma(2+a-b)\Gamma(1+a-c)\Gamma(\frac{2+a}{2}-b-c)}
{\Gamma(1+a)\Gamma(\frac{2+a}{2}-b)\Gamma(\frac{2+a}{2}-c)\Gamma(2+a-b-c)}\\
&&\:\:+\:
\frac{a(x+b-a-1)}{2x(b-1)}\frac{\Gamma(\frac{1+a}{2})\Gamma(2+a-b)\Gamma(1+a-c)\Gamma(\frac{3+a}{2}-b-c)}
{\Gamma(1+a)\Gamma(\frac{3+a}{2}-b)\Gamma(\frac{1+a}{2}-c)\Gamma(2+a-b-c)}.
 \enm
When the parameter $x$ is specified, the last equation can produce
the following two results:
 \bmn
 &&_3F_2\ffnk{cccc}{1}{a,b,c}{2+a-b,1+a-c}
  \nnm\\\nnm
&&\:=\:\frac{1}{1-b}\frac{\Gamma(\frac{2+a}{2})\Gamma(2+a-b)\Gamma(1+a-c)\Gamma(\frac{4+a}{2}-b-c)}
{\Gamma(1+a)\Gamma(\frac{2+a}{2}-b)\Gamma(\frac{2+a}{2}-c)\Gamma(2+a-b-c)}
     \\ \label{dixon-b}
&&\:\:-\:\frac{1}{2(1-b)}
\frac{\Gamma(\frac{1+a}{2})\Gamma(2+a-b)\Gamma(1+a-c)\Gamma(\frac{3+a}{2}-b-c)}
{\Gamma(a)\Gamma(\frac{3+a}{2}-b)\Gamma(\frac{1+a}{2}-c)\Gamma(2+a-b-c)},\\
 &&_3F_2\ffnk{cccc}{1}{a,b,c}{2+a-b,a-c}
  \nnm\\\nnm
&&\:=\:\frac{1-b-c}{1-b}\frac{\Gamma(\frac{2+a}{2})\Gamma(2+a-b)\Gamma(a-c)\Gamma(\frac{2+a}{2}-b-c)}
{\Gamma(1+a)\Gamma(\frac{2+a}{2}-b)\Gamma(\frac{a}{2}-c)\Gamma(2+a-b-c)}
\\ \label{dixon-c}
 &&\:\:+\:\frac{1-b+c}{2(1-b)}
\frac{\Gamma(\frac{1+a}{2})\Gamma(2+a-b)\Gamma(a-c)\Gamma(\frac{3+a}{2}-b-c)}
{\Gamma(a)\Gamma(\frac{3+a}{2}-b)\Gamma(\frac{1+a}{2}-c)\Gamma(2+a-b-c)}.
 \emn
It is routine to show the continuous relation
 \bnm
 _3F_2\ffnk{cccc}{1}{a,b,c}{3+a-b,a-c}&&\xqdn=\:\frac{(2+a-b)c}{(a-c)(b-2)}
 {_3F_2}\ffnk{cccc}{1}{a,b,c}{2+a-b,1+a-c}\\
&&\xqdn
+\:\:\frac{a(b-c-2)}{(a-c)(b-2)}{_3F_2}\ffnk{cccc}{1}{1+a,b,c}{3+a-b,1+a-c}.
 \enm
 Calculating, respectively, the two series on the right hand side by
 \eqref{dixon-b} and \eqref{dixon-c}, we obtain
\bmn\label{dixon-d}
 &&\xxqdn_3F_2\ffnk{cccc}{1}{a,b,c}{3+a-b,a-c}
  \nnm\\\nnm
&&\xxqdn\:=\:\frac{2+b^2-3b+bc-ac-c}{(b-1)(b-2)}
\frac{\Gamma(\frac{2+a}{2})\Gamma(3+a-b)\Gamma(a-c)\Gamma(\frac{4+a}{2}-b-c)}
{\Gamma(1+a)\Gamma(\frac{4+a}{2}-b)\Gamma(\frac{a}{2}-c)\Gamma(3+a-b-c)}\\
&&\xxqdn\:\:+\:\frac{2+b^2-3b-bc+ac+c-2c^2}{2(b-1)(b-2)}\frac{\Gamma(\frac{1+a}{2})\Gamma(3+a-b)\Gamma(a-c)\Gamma(\frac{3+a}{2}-b-c)}
{\Gamma(a)\Gamma(\frac{3+a}{2}-b)\Gamma(\frac{1+a}{2}-c)\Gamma(3+a-b-c)}.
 \emn
In accordance with \eqref{kummer}, we get
 \bnm
 _3F_2\ffnk{cccc}{1}{a,3-a,b}{c,2b-c}=\frac{\Gamma(2b-c)\Gamma(b-3)}{\Gamma(2b-a-c)\Gamma(a+b-3)}
 {_3F_2}\ffnk{cccc}{1}{a,c-b,a+c-3}{c,a+b-3}.
 \enm
Evaluating the series on the right hand side by \eqref{dixon-d}, we
gain
 \bmn\label{whipple-d}
_3F_2\ffnk{cccc}{1}{a,3-a,b}{c,2b-c}&&\xqdn=\:\frac{4(2-3a+a^2-2b+bc+2c-c^2)}{(a-1)(a-2)(b-1)(b-2)(b-3)}
\nnm\\
&&\xqdn\times\:\:\frac{\Gamma(\frac{c}{2})\Gamma(\frac{1+c}{2})\Gamma(b-\frac{c}{2})\Gamma(b-\frac{c-1}{2})}
{\Gamma(\frac{a+c-2}{2})\Gamma(\frac{1-a+c}{2})\Gamma(b-\frac{a+c}{2})\Gamma(b-\frac{3-a+c}{2})}
\nnm\\\nnm
&&\xqdn\:+\:\:\frac{4(2-3a+a^2+2b-2b^2+3bc-2c-c^2)}{(a-1)(a-2)(b-1)(b-2)(b-3)}\\
&&\xqdn\times\:\:\frac{\Gamma(\frac{c}{2})\Gamma(\frac{1+c}{2})\Gamma(b-\frac{c}{2})\Gamma(b-\frac{c-1}{2})}
{\Gamma(\frac{a+c-3}{2})\Gamma(\frac{c-a}{2})\Gamma(b-\frac{a+c-1}{2})\Gamma(b-\frac{2-a+c}{2})}.
 \emn
It is not difficult to verify that
 \bnm
&&\sum_{k=0}^{\infty}k\frac{(a)_k(1-a)_k(1+b)_k}{k!(1+c)_k(1+2b-c)_k}\\
&&\:=\:\sum_{k=1}^{\infty}\frac{(a)_k(1-a)_k(1+b)_k}{(k-1)!(1+c)_k(1+2b-c)_k}\\
&&\:=\:\sum_{k=0}^{\infty}\frac{(a)_{k+1}(1-a)_{k+1}(1+b)_{k+1}}{k!(1+c)_{k+1}(1+2b-c)_{k+1}}\\
&&\:=\:\frac{a(1-a)(1+b)}{(1+c)(1+2b-c)}{_3F_2}\ffnk{cccc}{1}{1+a,2-a,2+b}{2+c,2+2b-c}.
 \enm
Calculating the series on the right hand side by \eqref{whipple-d},
we achieve Lemma \ref{lemm-c}.
 \end{proof}

\begin{thm} \label{thm-b}
Let $x$ be a complex number. Then
 \bnm
&&\xxqdn\sum_{k=0}^n(-1)^k\binm{n}{k}\frac{\binm{n+k}{k}}{\binm{x+k}{k}}kH_{k}^{\langle2\rangle}(x)
=(-1)^n\frac{n(n+1)}{2(1-x)}\frac{\binm{-x+n}{n}}{\binm{x+n}{n}}\\
&&\xxqdn\:\:\times\:\bigg\{H_{n}^{\langle2\rangle}(x)-H_{n}^{\langle2\rangle}(-x)
+\Big[H_n(x)-H_{n}(-x)\Big]\\
&&\xxqdn\:\:\times\:\Big[H_n(x)-H_{n}(-x)-2H_{n+1}(\tfrac{x-n-2}{2})-\tfrac{2(x^2-n-n^2)}{xn(n+1)}\Big]\\
&&\xxqdn\:\:+\:\,H_{n}(\tfrac{x-n}{2})\Big[H_{n+1}(\tfrac{x-n-2}{2})+\tfrac{2(x^3-nx^2+n^2+n^3)}{x(x-n)n(n+1)}\Big]
+\tfrac{4(x^2-nx+n+n^2)}{x(x-n)^2(n+1)}\bigg\}.
 \enm
\end{thm}

\begin{proof}
The case $a=-n$, $b=x$ and $c=y$ of Lemma \ref{lemm-c} reads as
 \bmn\label{whipple-e}
&&\xxqdn\xqdn\sum_{k=0}^n(-1)^kk\binm{n}{k}\frac{\binm{n+k}{k}\binm{x+k}{k}}{\binm{y+k}{k}\binm{2x-y+k}{k}}
 \nnm\\\nnm
&&\xxqdn\xqdn\:=\:\frac{(2x-y)(n^2+n+xy-y^2)}{2x(1-x)}\frac{\binm{\frac{y-n-1}{2}+n}{n}\binm{y-2x+n}{n}}{\binm{\frac{y-n-1}{2}-x+n}{n}\binm{y+n}{n}}\\
&&\xxqdn\xqdn\:\,+\:\:\frac{(2x-y)(y-n)(n^2+n-2x^2+3xy-y^2)}{2x(1-x)(2x-y+n)}\frac{\binm{\frac{y-n}{2}+n}{n}\binm{y-2x+n}{n}}{\binm{\frac{y-n}{2}-x+n}{n}\binm{y+n}{n}}.
 \emn
 Applying the derivative operator $\mathcal{D}_y$ to both sides of
 it, we attain
 \bnm
\sum_{k=0}^n(-1)^k\binm{n}{k}\frac{\binm{n+k}{k}\binm{x+k}{k}}{\binm{y+k}{k}\binm{2x-y+k}{k}}
k\big\{H_k(2x-y)-H_k(y)\big\}=\Phi_n(x,y),
 \enm
where the symbol on the right hand side stands for
 \bnm
\Phi_n(x,y)&&\xqdn\!=\frac{(2x-y)(n^2+n+xy-y^2)}{2x(1-x)}
\frac{\binm{\frac{y-n-1}{2}+n}{n}\binm{y-2x+n}{n}}{\binm{\frac{y-n-1}{2}-x+n}{n}\binm{y+n}{n}}\Big\{\tfrac{1}{2}H_n(\tfrac{y-n-1}{2})\\
&&\xqdn\!-\:\tfrac{1}{2}H_n(\tfrac{y-n-1}{2}-x)+H_{n+1}(y-2x-1)-H_n(y)+\tfrac{x-2y}{n^2+n+xy-y^2}\Big\}\\
&&\xqdn\!+\:\frac{(2x-y)(y-n)(n^2+n-2x^2+3xy-y^2)}{2x(1-x)(2x-y+n)}
\frac{\binm{\frac{y-n}{2}+n}{n}\binm{y-2x+n}{n}}{\binm{\frac{y-n}{2}-x+n}{n}\binm{y+n}{n}}\Big\{\tfrac{1}{2}H_{n+1}(\tfrac{y-n-2}{2})\\
&&\xqdn\!-\:\tfrac{1}{2}H_{n+1}(\tfrac{y-n-2}{2}-x)+H_{n+1}(y-2x-1)-H_n(y)
+\tfrac{3x-2y}{n^2+n-2x^2+3xy-y^2}\Big\}.
 \enm
 The last equation can be written as
 \bnm
\sum_{k=0}^n(-1)^k\binm{n}{k}\frac{\binm{n+k}{k}\binm{x+k}{k}}{\binm{y+k}{k}\binm{2x-y+k}{k}}
k\sum_{i=1}^k\frac{1}{(2x-y+i)(y+i)}=\frac{\Phi_n(x,y)}{2(y-x)}.
 \enm
Taking the limit $y\to x$ of it by utilizing the relation
 \bnm
&&\text{Lim}_{y\to x}\frac{\Phi_n(x,y)}{2(y-x)}\\
&&\:=\:\text{Lim}_{y\to x}
\frac{\mathcal{D}_y\Phi_n(x,y)}{2}\\
&&\:=\:(-1)^n\frac{n(n+1)}{2(1-x)}\frac{\binm{-x+n}{n}}{\binm{x+n}{n}}\\
&&\:\:\times\:\bigg\{H_{n}^{\langle2\rangle}(x)-H_{n}^{\langle2\rangle}(-x)
+\Big[H_n(x)-H_{n}(-x)\Big]\\
&&\:\:\times\:\Big[H_n(x)-H_{n}(-x)-2H_{n+1}(\tfrac{x-n-2}{2})-\tfrac{2(x^2-n-n^2)}{xn(n+1)}\Big]\\
&&\:\:+\:\,H_{n}(\tfrac{x-n}{2})\Big[H_{n+1}(\tfrac{x-n-2}{2})+\tfrac{2(x^3-nx^2+n^2+n^3)}{x(x-n)n(n+1)}\Big]
+\tfrac{4(x^2-nx+n+n^2)}{x(x-n)^2(n+1)}\bigg\}
 \enm
from L'H\^{o}spital rule, we obtain Theorem \ref{thm-b}.
\end{proof}

 When $x\to p$, where $p$ a nonnegative integer, Theorem \ref{thm-b}
 can give the following four corollaries.

\begin{corl}[Harmonic number identity] \label{corl-e}
\bnm
 &&\xxqdn\sum_{k=0}^n(-1)^k\binm{n}{k}\binm{n+k}{k}kH_{k}^{\langle2\rangle}=(-1)^nn(n+1)\\
 &&\xxqdn\:\times\:\begin{cases}
 2H_{n}^{\langle2\rangle}-H_{\frac{n}{2}}^{\langle2\rangle},&n=0\,(\qqdn\mod2);\\[2mm]
 2H_{n}^{\langle2\rangle}-H_{\frac{n-1}{2}}^{\langle2\rangle}-\frac{2}{n(n+1)},&n=1\,(\qqdn\mod2).
\end{cases}
 \enm
\end{corl}

\begin{corl} \label{corl-f}
Let $p$ be a positive integer satisfying $0<p\leq n$. Then
 \bnm
 &&\sum_{k=0}^n(-1)^k\binm{n+k}{k}\binm{p+n}{n-k}kH_{p+k}^{\langle2\rangle}=\frac{n(n+1)}{p(p-1)}\frac{(-1)^{n-p}}{\binm{n}{p}}\\
 &&\:\times\:\begin{cases}
 H_{n+p}- H_{n-p}-H_{\frac{n+p}{2}}+H_{\frac{n-p}{2}}-\frac{p}{n(n+1)},&n-p=0\,(\qqdn\mod2);\\[2mm]
  H_{n+p}-H_{n-p}-H_{\frac{n+p-1}{2}}+H_{\frac{n-p-1}{2}}+\frac{p}{n(n+1)},&n-p=1\,(\qqdn\mod2).
\end{cases}
 \enm
\end{corl}

\begin{corl}[Harmonic number identity] \label{corl-g}
 \bnm
 \qquad\sum_{k=0}^n(-1)^k\binm{n}{k}kH_{n+k}^{\langle2\rangle}=\frac{n+1}{n-1}\frac{1}{\binm{2n}{n}}\Big\{H_{2n}-H_{n+1}\Big\}.
 \enm
\end{corl}

\begin{corl} \label{corl-h}
Let $p$ be a positive integer with $p>n$. Then
 \bnm
\quad
\sum_{k=0}^n(-1)^k\binm{n+k}{k}\binm{p+n}{n-k}kH_{p+k}^{\langle2\rangle}=\frac{n+n^2}{2(1-p)}\binm{p-1}{n}
 \Big\{H_{p+n}^{\langle2\rangle}+H_{p-n-1}^{\langle2\rangle}+B(p,n)\Big\},
 \enm
where the symbol on the right hand side stands for
 \bnm
 \:\:B(p,n)=\begin{cases}
(H_{p+n}- H_{p-n}-H_{\frac{p+n}{2}}+H_{\frac{p-n}{2}}-\frac{3}{p-n}-\frac{2p}{n+n^2})\\
\times(H_{p+n}- H_{p-n-1}-H_{\frac{p+n}{2}}+H_{\frac{p-n}{2}})+\frac{2(p^2+2n+n^2)}{(n+n^2)(p-n)^2},&p-n=0\,(\qqdn\mod2);\\[2mm]
  (H_{p+n}- H_{p-n-1}-H_{\frac{p+n-1}{2}}+H_{\frac{p-n-1}{2}}+\frac{2p}{n+n^2})\\
\times(H_{p+n}-H_{p-n-1}-H_{\frac{p+n-1}{2}}+H_{\frac{p-n-1}{2}})-\frac{2}{n+n^2},&p-n=1\,(\qqdn\mod2).
\end{cases}
 \enm
\end{corl}

\begin{lemm}\label{lemm-d}
Let $a$, $b$ and $c$ be all complex numbers. Then
 \bnm
&&\xqdn\xxqdn\sum_{k=0}^{\infty}k^2\frac{(a)_k(1-a)_k(1+b)_k}{k!(1+c)_k(1+2b-c)_k}\\
&&\xqdn\xxqdn\:=\:\frac{\beta(a,b,c)}{b(b-1)(b-2)}\frac{\Gamma(\frac{1+c}{2})\Gamma(\frac{2+c}{2})\Gamma(b+\frac{1-c}{2})\Gamma(b+\frac{2-c}{2})}
{\Gamma(\frac{1+a+c}{2})\Gamma(\frac{2-a+c}{2})\Gamma(b+\frac{a-c}{2})\Gamma(b+\frac{1-a-c}{2})}\\
&&\xqdn\xxqdn\:\:+\:\,\frac{\gamma(a,b,c)}{b(b-1)(b-2)}\frac{\Gamma(\frac{1+c}{2})\Gamma(\frac{2+c}{2})\Gamma(b+\frac{1-c}{2})\Gamma(b+\frac{2-c}{2})}
{\Gamma(\frac{a+c}{2})\Gamma(\frac{1-a+c}{2})\Gamma(b+\frac{1+a-c}{2})\Gamma(b+\frac{2-a-c}{2})},
 \enm
where the convergence condition is $Re(b-2)>0$ and
 \bnm
\beta(a,b,c)&&\xqdn=\:c(c-b)(1+b-2bc+c^2)+a(b-3bc+b^2+2c^2)\\
&&\xqdn+\:\,a^2(1-b+3bc-b^2-2c^2)-2a^3+a^4,\\
\gamma(a,b,c)&&\xqdn=\:(c-b)(c-2b)(1+b-2bc+c^2)+a(b-5bc+3b^2+2c^2)\\
&&\xqdn+\:\,a^2(1-b+5bc-3b^2-2c^2)-2a^3+a^4.
 \enm
\end{lemm}

\begin{proof}
It is easy to see the continuous relation
  \bnm
\quad
_3F_2\ffnk{cccc}{1}{a,b,c}{3+a-b,1+a-c}&&\xqdn=\:\frac{2+a-b}{2-b}
 {_3F_2}\ffnk{cccc}{1}{a,b,c}{2+a-b,1+a-c}\\
&&\xqdn
+\:\:\frac{a}{b-2}{_3F_2}\ffnk{cccc}{1}{1+a,b,c}{3+a-b,1+a-c}.
 \enm
 Evaluating, respectively, the two series on the right hand side by
 \eqref{dixon-b} and \eqref{dixon-c}, we get
\bmn\label{dixon-e}
 &&\xxqdn\xxqdn_3F_2\ffnk{cccc}{1}{a,b,c}{3+a-b,1+a-c}
  \nnm\\\nnm
&&\xxqdn\xxqdn\:=\:\frac{\sst a^2-a(2b+2c-3)+2(b+c-2)}{2(b-1)(b-2)}
\frac{\Gamma(\frac{2+a}{2})\Gamma(3+a-b)\Gamma(1+a-c)\Gamma(\frac{4+a}{2}-b-c)}
{\Gamma(1+a)\Gamma(\frac{4+a}{2}-b)\Gamma(\frac{2+a}{2}-c)\Gamma(3+a-b-c)}\\
&&\xxqdn\xxqdn\:\:-\:\frac{1}{(b-1)(b-2)}\frac{\Gamma(\frac{1+a}{2})\Gamma(3+a-b)\Gamma(1+a-c)\Gamma(\frac{5+a}{2}-b-c)}
{\Gamma(a)\Gamma(\frac{3+a}{2}-b)\Gamma(\frac{1+a}{2}-c)\Gamma(3+a-b-c)}.
 \emn
It is routine to show the relation
 \bnm\:\:
 _4F_3\ffnk{cccc}{1}{a,b,c,1+x}{4+a-b,1+a-c,x}&&\xqdn\!=\frac{a(3+a-b-x)}{(3-b)x}
 {_3F_2}\ffnk{cccc}{1}{1+a,b,c}{4+a-b,1+a-c}\\
 &&\xqdn\!+\:\frac{(3+a-b)(x-a)}{(3-b)x}{_3F_2}\ffnk{cccc}{1}{a,b,c}{3+a-b,1+a-c}.
 \enm
Calculating, respectively, the two series on the right hand side by
 \eqref{dixon-d} and \eqref{dixon-e}, we gain
 \bnm
 &&\xqdn_4F_3\ffnk{cccc}{1}{a,b,c,1+x}{4+a-b,1+a-c,x}\\
&&\xqdn\:=\:\frac{\begin{cases} a(3+a-b-x)(2-3b+2c+b^2+ac-bc-2c^2)+(x-a)\\
\times(3+a-b-c)(4+3a-6b-2c+a^2+2b^2-2ab-2ac+2bc)
\end{cases}}{2x(1-b)(2-b)(3-b)}\\
&&\xqdn\:\:\times\:\frac{\Gamma(\frac{2+a}{2})\Gamma(4+a-b)\Gamma(1+a-c)\Gamma(\frac{4+a}{2}-b-c)}
{\Gamma(1+a)\Gamma(\frac{4+a}{2}-b)\Gamma(\frac{2+a}{2}-c)\Gamma(4+a-b-c)}\\
&&\xqdn\:\:+\:
\frac{(a-x)(3+a-2b)(3+a-b-c)+(3+a-b-x)(2-3b-2c+b^2-ac+bc)}{2x(1-b)(2-b)(3-b)}\\
&&\xqdn\:\:\times\:\frac{\Gamma(\frac{1+a}{2})\Gamma(4+a-b)\Gamma(1+a-c)\Gamma(\frac{5+a}{2}-b-c)}
{\Gamma(a)\Gamma(\frac{5+a}{2}-b)\Gamma(\frac{1+a}{2}-c)\Gamma(4+a-b-c)}.
 \enm
When the parameter $x$ is specified, the last equation can offer the
following two results:
 \bmn
 &&\xxqdn_3F_2\ffnk{cccc}{1}{a,b,c}{4+a-b,1+a-c}
  \nnm\\\nnm
&&\xxqdn\:=\:\frac{3+2a+a^2-4b+b^2-c-ab-2ac+bc}{(1-b)(2-b)(3-b)}\\
 &&\xxqdn\:\:\times\:\frac{\Gamma(\frac{2+a}{2})\Gamma(4+a-b)\Gamma(1+a-c)\Gamma(\frac{6+a}{2}-b-c)}
{\Gamma(1+a)\Gamma(\frac{4+a}{2}-b)\Gamma(\frac{2+a}{2}-c)\Gamma(4+a-b-c)}
     \nnm\\\nnm
&&\xxqdn\:\:-\:\frac{11+6a+a^2-12b+3b^2-5c-3ab-2ac+3bc}{2(1-b)(2-b)(3-b)}
    \\\label{dixon-f}
 &&\xxqdn\:\:\times\:\frac{\Gamma(\frac{1+a}{2})\Gamma(4+a-b)\Gamma(1+a-c)\Gamma(\frac{5+a}{2}-b-c)}
{\Gamma(a)\Gamma(\frac{5+a}{2}-b)\Gamma(\frac{1+a}{2}-c)\Gamma(4+a-b-c)},\\
 &&\xxqdn_3F_2\ffnk{cccc}{1}{a,b,c}{4+a-b,a-c}
  \nnm\\\nnm
&&\xxqdn\:=\:\frac{b^3+2b^2(c-3)+b(11-7c-2ac+c^2)+c(5+3a+a^2-c-2ac)-6}{(b-1)(b-2)(b-3)}\\
 &&\xxqdn\:\:\times\:\frac{\Gamma(\frac{2+a}{2})\Gamma(4+a-b)\Gamma(a-c)\Gamma(\frac{4+a}{2}-b-c)}
{\Gamma(1+a)\Gamma(\frac{4+a}{2}-b)\Gamma(\frac{a}{2}-c)\Gamma(4+a-b-c)}
 \nnm
 \emn
 \bmn
 \nnm
&&\xxqdn\:\:+\:\frac{b^3-2b^2(c+3)+b(11+7c+2ac-3c^2)+c^2(5+2a)-c(5+3a+a^2)-6}{2(b-1)(b-2)(b-3)}
    \\\label{dixon-g}
&&\xxqdn\:\:\times\:\frac{\Gamma(\frac{1+a}{2})\Gamma(4+a-b)\Gamma(a-c)\Gamma(\frac{5+a}{2}-b-c)}
{\Gamma(a)\Gamma(\frac{5+a}{2}-b)\Gamma(\frac{1+a}{2}-c)\Gamma(4+a-b-c)}.
 \emn
It is not difficult to verify the continuous relation
 \bnm
\xqdn\qqdn
_3F_2\ffnk{cccc}{1}{a,b,c}{5+a-b,a-c}&&\xqdn=\:\frac{(4+a-b)c}{(a-c)(b-4)}
 {_3F_2}\ffnk{cccc}{1}{a,b,c}{4+a-b,1+a-c}\\
&&\xqdn
+\:\:\frac{a(b-c-4)}{(a-c)(b-4)}{_3F_2}\ffnk{cccc}{1}{1+a,b,c}{5+a-b,1+a-c}.
 \enm
 Evaluating, respectively, the two series on the right hand side by
 \eqref{dixon-f} and \eqref{dixon-g}, we have
\bmn\label{dixon-h}
 &&\xxqdn_3F_2\ffnk{cccc}{1}{a,b,c}{5+a-b,a-c}
 \nnm\\\nnm
&&\xxqdn\:=\:\frac{\begin{cases}
 \sst a(b-c-4)[b^3-2b^2(c+3)+b(11+9c+2ac-3c^2)-c(9+5a+a^2-7c-2ac)-6]\\
\sst-c(4+a-2b)(4+a-b-c)[a^2-a(b+2c-2)+(b-1)(b+c-3)]
\end{cases}}{2(b-1)(b-2)(b-3)(b-4)}\\
&&\xxqdn\:\:\times\:\:\frac{\Gamma(\frac{2+a}{2})\Gamma(5+a-b)\Gamma(a-c)\Gamma(\frac{6+a}{2}-b-c)}
{\Gamma(1+a)\Gamma(\frac{6+a}{2}-b)\Gamma(\frac{2+a}{2}-c)\Gamma(5+a-b-c)}
\nnm\\\nnm
 &&\xxqdn\:\:+\:
\frac{\begin{cases}
 \sst b^4-10b^3+35b^2-50b+24+c(2+a-b)(7+4a+a^2-9b-2ab+2b^2)\\
\sst-(28+16a+4a^2-31b-8ab+7b^2)c^2+4(2+a-b)c^3
\end{cases}}{2(b-1)(b-2)(b-3)(b-4)}\\
&&\xxqdn\:\:\times\:\:\frac{\Gamma(\frac{1+a}{2})\Gamma(5+a-b)\Gamma(a-c)\Gamma(\frac{5+a}{2}-b-c)}
{\Gamma(a)\Gamma(\frac{5+a}{2}-b)\Gamma(\frac{1+a}{2}-c)\Gamma(5+a-b-c)}.
 \emn
In terms of \eqref{kummer}, we achieve
 \bnm\quad\:
 _3F_2\ffnk{cccc}{1}{a,5-a,b}{c,2b-c}=\frac{\Gamma(2b-c)\Gamma(b-5)}{\Gamma(2b-a-c)\Gamma(a+b-5)}
 {_3F_2}\ffnk{cccc}{1}{a,c-b,a+c-5}{c,a+b-5}.
 \enm
Calculating the series on the right hand side by \eqref{dixon-h}, we
attain
 \bmn\label{whipple-f}
&&\qqdn_3F_2\ffnk{cccc}{1}{a,5-a,b}{c,2b-c}
 \nnm\\\nnm
&&\qqdn\:=\: \frac{\begin{cases}
 a^4-10a^3+35a^2-50a+24-(b-c)(36-15a-12b-18c\\
-5ab+10ac+12bc+3a^2-3c^2+a^2b-2a^2c-2bc^2+c^3)
\end{cases}}{(a-1)(a-2)(a-3)(a-4)(b-1)(b-2)(b-3)(b-4)(b-5)}\\
&&\qqdn\:\:\times\:\:\frac{16\Gamma(\frac{c}{2})\Gamma(\frac{1+c}{2})\Gamma(b-\frac{c}{2})\Gamma(b-\frac{c-1}{2})}
{\Gamma(\frac{a+c-4}{2})\Gamma(\frac{1-a+c}{2})\Gamma(b-\frac{a+c}{2})\Gamma(b-\frac{5-a+c}{2})}
\nnm\\\nnm &&\qqdn\:\:+\:\:\frac{\begin{cases}
 a^4-10a^3+35a^2-50a+24-(a^2-5a+4b^2-8bc+12-6c+5c^2)\\
(b-c)(c-3)-(b-c)^2(3a^2-15a+48-24c+4c^2)
\end{cases}}{(a-1)(a-2)(a-3)(a-4)(b-1)(b-2)(b-3)(b-4)(b-5)}\\
&&\qqdn\:\:\times\:\:\frac{16\Gamma(\frac{c}{2})\Gamma(\frac{1+c}{2})\Gamma(b-\frac{c}{2})\Gamma(b-\frac{c-1}{2})}
{\Gamma(\frac{a+c-5}{2})\Gamma(\frac{c-a}{2})\Gamma(b-\frac{a+c-1}{2})\Gamma(b-\frac{4-a+c}{2})}.
 \emn
It is easy to see that
 \bnm
&&\sum_{k=0}^{\infty}k^2\frac{(a)_k(1-a)_k(1+b)_k}{k!(1+c)_k(1+2b-c)_k}\\
&&\:=\:\sum_{k=1}^{\infty}k\frac{(a)_k(1-a)_k(1+b)_k}{(k-1)!(1+c)_k(1+2b-c)_k}\\
&&\:=\:\sum_{k=0}^{\infty}(1+k)\frac{(a)_{k+1}(1-a)_{k+1}(1+b)_{k+1}}{k!(1+c)_{k+1}(1+2b-c)_{k+1}}\\
&&\:=\:\frac{a(1-a)(1+b)}{(1+c)(1+2b-c)}{_3F_2}\ffnk{cccc}{1}{1+a,2-a,2+b}{2+c,2+2b-c}\\
&&\:\:+\:\frac{a(1+a)(1-a)(2-a)(1+b)(2+b)}{(1+c)(2+c)(1+2b-c)(2+2b-c)}{_3F_2}\ffnk{cccc}{1}{2+a,3-a,3+b}{3+c,3+2b-c}.
 \enm
Evaluating, respectively, the two series on the right hand side by
\eqref{whipple-d} and \eqref{whipple-f}, we obtain Lemma
\ref{lemm-d}.
\end{proof}

\begin{thm} \label{thm-c}
Let $x$ be a complex number. Then
 \bnm
&&\xxqdn\sum_{k=0}^n(-1)^k\binm{n}{k}\frac{\binm{n+k}{k}}{\binm{x+k}{k}}k^2H_{k}^{\langle2\rangle}(x)
=(-1)^n\frac{(n+n^2)(n+n^2-x)}{2(1-x)(2-x)}\frac{\binm{-x+n}{n}}{\binm{x+n}{n}}\\
&&\xxqdn\:\:\times\:\bigg\{H_{n}^{\langle2\rangle}(x)-H_{n}^{\langle2\rangle}(-x)
+\Big[H_n(x)-H_{n}(-x)-2H_{n+1}(\tfrac{x-n-2}{2})+U_n(x)\Big]\\
&&\xxqdn\:\:\times\:\Big[H_n(x)-H_{n}(-x)\Big]
+H_{n}(\tfrac{x-n}{2})\Big[H_{n}(\tfrac{x-n}{2})-V_n(x)\Big]+W_n(x)\bigg\},
 \enm
where the expressions on the right hand side are
 \bnm
&&\xqdn U_n(x)=\frac{2[(n+n^2)(n+n^2-x)+(1-n-n^2)x^2+x^3-x^4]}{(n+n^2)(n+n^2-x)x},\\
&&\xqdn V_n(x)=\frac{2[n^3(1+n)(1+n+x)-n^2(2+n)x^2-(1-2n-n^2)x^3-(1+n)x^4+x^5]}{(n+n^2)(n+n^2-x)(n-x)x},\\
&&\xqdn
W_n(x)=\frac{2[2n^2(1+n)^2-n^2(5+3n)x-2(1-3n-2n^2)x^2-3(1+n)x^3+2x^4]}{(1+n)(n+n^2-x)(n-x)^2x}.
 \enm
\end{thm}

 \begin{proof}
The case $a=-n$, $b=x$ and $c=y$ of Lemma \ref{lemm-d} reads as
 \bmn\label{whipple-g}
\quad\sum_{k=0}^n(-1)^kk^2\binm{n}{k}\frac{\binm{n+k}{k}\binm{x+k}{k}}{\binm{y+k}{k}\binm{2x-y+k}{k}}=\Psi_n(x,y),
 \emn
where the symbol on the right hand side stands for
 \bnm
\Psi_n(x,y)&&\xqdn\!=\frac{\begin{cases}
 y(y-x)(1+x-2xy+y^2)-n(x-3xy+x^2+2y^2)\\
+n^2(1-x+3xy-x^2-2y^2)+2n^3+n^4
\end{cases}}{2x(x-1)(x-2)}\\
&&\xqdn\!\times\:(2x-y)\frac{\binm{\frac{y-n-1}{2}+n}{n}\binm{y-2x+n}{n}}{\binm{\frac{y-n-1}{2}-x+n}{n}\binm{y+n}{n}}
+\frac{\binm{\frac{y-n}{2}+n}{n}\binm{y-2x+n}{n}}{\binm{\frac{y-n}{2}-x+n}{n}\binm{y+n}{n}}(2x-y)(y-n)\\
&&\xqdn\!\times\:\frac{\begin{cases}
 (x-y)(2x-y)(1+x-2xy+y^2)-n(x-5xy+3x^2+2y^2)\\
+n^2(1-x+5xy-3x^2-2y^2)+2n^3+n^4
\end{cases}}{2x(x-1)(x-2)(2x-y+n)}.
 \enm
 Applying the derivative operator $\mathcal{D}_y$ to both sides of
 it, we get
 \bnm
\sum_{k=0}^n(-1)^k\binm{n}{k}\frac{\binm{n+k}{k}\binm{x+k}{k}}{\binm{y+k}{k}\binm{2x-y+k}{k}}
k^2\big\{H_k(2x-y)-H_k(y)\big\}=\mathcal{D}_y\Psi_n(x,y).
 \enm
The last equation can be reformulated as
 \bnm
\:\sum_{k=0}^n(-1)^k\binm{n}{k}\frac{\binm{n+k}{k}\binm{x+k}{k}}{\binm{y+k}{k}\binm{2x-y+k}{k}}
k^2\sum_{i=1}^k\frac{1}{(2x-y+i)(y+i)}=\frac{\mathcal{D}_y\Psi_n(x,y)}{2(y-x)}.
 \enm
Finding the limit $y\to x$ of it by using the relation
 \bnm
&&\xqdn\qqdn\text{Lim}_{y\to x}\frac{\mathcal{D}_y\Psi_n(x,y)}{2(y-x)}\\
&&\xqdn\qqdn\:=\:\text{Lim}_{y\to x}
\frac{\mathcal{D}_y^2\Psi_n(x,y)}{2}\\
&&\xqdn\qqdn\:=\:(-1)^n\frac{(n+n^2)(n+n^2-x)}{2(1-x)(2-x)}\frac{\binm{-x+n}{n}}{\binm{x+n}{n}}\\
&&\xqdn\qqdn\:\:\times\:\bigg\{H_{n}^{\langle2\rangle}(x)-H_{n}^{\langle2\rangle}(-x)
+\Big[H_n(x)-H_{n}(-x)-2H_{n+1}(\tfrac{x-n-2}{2})+U_n(x)\Big]\\
&&\xqdn\qqdn\:\:\times\:\Big[H_n(x)-H_{n}(-x)\Big]
+H_{n}(\tfrac{x-n}{2})\Big[H_{n}(\tfrac{x-n}{2})-V_n(x)\Big]+W_n(x)\bigg\}
 \enm
from L'H\^{o}spital rule, we gain Theorem \ref{thm-c}.

\end{proof}

When $x\to p$, where $p$ a nonnegative integer, Theorem \ref{thm-c}
 can produce the following four corollaries.

\begin{corl}[Harmonic number identity] \label{corl-i}
\bnm
 &&\xxqdn\sum_{k=0}^n(-1)^k\binm{n}{k}\binm{n+k}{k}k^2H_{k}^{\langle2\rangle}=(-1)^n\frac{n^2(n+1)^2}{2}\\
 &&\xxqdn\:\times\:\begin{cases}
 2H_{n}^{\langle2\rangle}-H_{\frac{n}{2}}^{\langle2\rangle}-\frac{1}{n(n+1)},&n=0\,(\qqdn\mod2);\\[2mm]
 2H_{n}^{\langle2\rangle}-H_{\frac{n-1}{2}}^{\langle2\rangle}+\frac{2-3n-3n^2}{n^2(n+1)^2},&n=1\,(\qqdn\mod2).
\end{cases}
 \enm
\end{corl}

\begin{corl} \label{corl-j}
Let $p$ be a positive integer satisfying $0<p\leq n$. Then
 \bnm
 &&\xqdn\sum_{k=0}^n(-1)^k\binm{n+k}{k}\binm{p+n}{n-k}k^2H_{p+k}^{\langle2\rangle}
 =\frac{(n+n^2)(n+n^2-p)}{p(p-1)(p-2)}\frac{(-1)^{n-p+1}}{\binm{n}{p}}\\
 &&\xqdn\:\times\:\begin{cases}
 H_{n+p}- H_{n-p}-H_{\frac{n+p}{2}}+H_{\frac{n-p}{2}}+\frac{p(1-n-n^2+p-p^2)}{(n+n^2)(n+n^2-p)},&n-p=0\,(\qqdn\mod2);\\[2mm]
  H_{n+p}-H_{n-p}-H_{\frac{n+p-1}{2}}+H_{\frac{n-p-1}{2}}-\frac{p(1-n-n^2+p-p^2)}{(n+n^2)(n+n^2-p)},&n-p=1\,(\qqdn\mod2).
\end{cases}
 \enm
\end{corl}

\begin{corl}[Harmonic number identity]\label{corl-k}
 \bnm
\:\:\:
\sum_{k=0}^n(-1)^k\binm{n}{k}k^2H_{n+k}^{\langle2\rangle}=\frac{n^2(n+1)}{(n-1)(n-2)}\frac{1}{\binm{2n}{n}}\Big\{H_{n+1}-H_{2n}+\tfrac{n-1}{n^2}\Big\}.
 \enm
\end{corl}

\begin{corl} \label{corl-l}
Let $p$ be a positive integer with $p>n$. Then
 \bnm
&&\sum_{k=0}^n(-1)^k\binm{n+k}{k}\binm{p+n}{n-k}k^2H_{p+k}^{\langle2\rangle}=\frac{(n+n^2)(n+n^2-p)}{2(p-1)(p-2)}\binm{p-1}{n}\\
 &&\:\times\:\Big\{H_{p+n}^{\langle2\rangle}+H_{p-n}^{\langle2\rangle}+C(p,n)\big[C(p,n)+D(p,n)\big]
  +\tfrac{2[1-2n(1+n)+(2+n)p-p^2]}{(1+n)(n+n^2-p)(n-p)}\Big\},
 \enm
where the symbols on the right hand side stand for
 \bnm
 &&\xqdn C(p,n)=\begin{cases}
   H_{p+n}- H_{p-n}-H_{\frac{p+n}{2}}+H_{\frac{p-n}{2}},&p-n=0\,(\qqdn\mod2);\\[2mm]
  H_{p-n}-H_{p+n}-H_{\frac{p-n-1}{2}}+H_{\frac{p+n-1}{2}},&p-n=1\,(\qqdn\mod2),
\end{cases}\\
&&\xqdn
D(p,n)=\frac{2[n^2(1+n)^2-n^2(2+n)p-(1-2n-n^2)p^2-(1+n)p^3+p^4]}{(n+n^2)(n+n^2-p)(n-p)}.
 \enm
\end{corl}
\section{The second family of summation formulae involving\\ generalized harmonic numbers}
\begin{thm} \label{thm-d}
Let $x$ be a complex number. Then
 \bnm
&&\xxqdn\sum_{k=0}^n(-1)^k\binm{n}{k}\frac{\binm{n+k}{k}}{\binm{x+k}{k}}H_{k}^{2}(x)
=\frac{(-1)^n}{2}\frac{\binm{-x+n}{n}}{\binm{x+n}{n}}\bigg\{H_{n}^{\langle2\rangle}(x)-H_{n}^{\langle2\rangle}(-x)\\
&&\xxqdn\:\:+\:2\Big[H_n(x)+H_{n}(-x)\Big]^2-\Big[H_n(x)-H_{n}(-x)-H_n(\tfrac{x-n}{2})\Big]\\
&&\xxqdn\:\:\times\:\Big[H_n(x)-H_{n}(-x)-H_n(\tfrac{x-n}{2})-\tfrac{2(x+n)}{x(x-n)}\Big]-\tfrac{4n}{x(x-n)^2}\bigg\}.
 \enm
\end{thm}

\begin{proof}
Fix $y=x$ in \eqref{whipple-b} to achieve
 \bmn\label{whipple-h}
\sum_{k=0}^n(-1)^k\binm{n}{k}\frac{\binm{n+k}{k}}{\binm{x+k}{k}}
=(-1)^n\frac{\binm{-x+n}{n}}{\binm{x+n}{n}}.
 \emn
 Applying the derivative operator $\mathcal{D}_y^2$ to both sides of
 it, we attain
\bnm
&&\xxqdn\qqdn\sum_{k=0}^n(-1)^k\binm{n}{k}\frac{\binm{n+k}{k}}{\binm{x+k}{k}}
 \Big\{H_{k}^{2}(x)+H_{k}^{\langle2\rangle}(x)\Big\}\\
&&\xxqdn\qqdn\:=\:(-1)^n\frac{\binm{-x+n}{n}}{\binm{x+n}{n}}
 \Big\{\big[H_{n}(x)+H_{n}(-x)\big]^2+H_{n}^{\langle2\rangle}(x)-H_{n}^{\langle2\rangle}(-x)\Big\}.
 \enm
The difference of the last equation and Theorem \ref{thm-a} gives
Theorem \ref{thm-d}.
\end{proof}

When $x\to p$, where $p$ a nonnegative integer, we can derive the
following four corollaries from Theorem \ref{thm-d}.

\begin{corl}[Harmonic number identity]\label{corl-m}
\bnm
 \qquad\sum_{k=0}^n(-1)^k\binm{n}{k}\binm{n+k}{k}H_{k}^2=
 \begin{cases}
 (-1)^n\Big\{4H_{n}^2-2H_{n}^{\langle2\rangle}+H_{\frac{n}{2}}^{2}\Big\},&n=0\,(\qqdn\mod2);\\[2mm]
  (-1)^n\Big\{4H_{n}^2-2H_{n}^{\langle2\rangle}+H_{\frac{n-1}{2}}^{\langle2\rangle}\Big\},&n=1\,(\qqdn\mod2).
\end{cases}
 \enm
\end{corl}

\begin{corl} \label{corl-n}
Let $p$ be a positive integer satisfying $0<p\leq n$. Then
 \bnm
 &&\sum_{k=0}^n(-1)^k\binm{n+k}{k}\binm{p+n}{n-k}H_{p+k}^2=\frac{(-1)^{n-p+1}}{p\binm{n}{p}}\\
 &&\:\times\:\begin{cases}
 H_{n+p}+3H_{n-p}-2H_{p-1}+H_{\frac{n+p}{2}}-H_{\frac{n-p}{2}},&n-p=0\,(\qqdn\mod2);\\[2mm]
  H_{n+p}+3H_{n-p}-2H_{p-1}+H_{\frac{n+p-1}{2}}-H_{\frac{n-p-1}{2}},&n-p=1\,(\qqdn\mod2).
\end{cases}
 \enm
\end{corl}

\begin{corl}[Harmonic number identity] \label{corl-o}
 \bnm
 \qquad\qquad\sum_{k=0}^n(-1)^k\binm{n}{k}H_{n+k}^2=\frac{1}{n\binm{2n}{n}}\Big\{H_{n}-H_{2n}-\tfrac{2}{n}\Big\}.
 \enm
\end{corl}

\begin{corl} \label{corl-p}
Let $p$ be a positive integer with $p>n$. Then
 \bnm
 &&\xxqdn\sum_{k=0}^n(-1)^k\binm{n+k}{k}\binm{p+n}{n-k}H_{p+k}^2=\frac{1}{2}\binm{p-1}{n}\\
&&\xxqdn\:\times\:\bigg\{H_{p+n}^{\langle2\rangle}+H_{p-n}^{\langle2\rangle}-2H_{p}^{\langle2\rangle}
+2H_p^2-\tfrac{4n}{p(p-n)}H_p-\tfrac{4n}{p^2(p-n)}\\
&&\xxqdn\:+\:2\Big[H_{p+n}+ H_{p-n}-2H_{p}\Big]\Big[H_{p+n}+
H_{p-n}-\tfrac{2n}{p(p-n)}\Big]-A(p,n)\bigg\},
 \enm
where the expression $A(p,n)$ on the right hand side has appeared in
Corollary \ref{corl-d}.
\end{corl}

\begin{thm} \label{thm-e}
Let $x$ be a complex number. Then
 \bnm
&&\xxqdn\sum_{k=0}^n(-1)^k\binm{n}{k}\frac{\binm{n+k}{k}}{\binm{x+k}{k}}kH_{k}^{2}(x)
=(-1)^n\frac{n(n+1)}{2(1-x)}\frac{\binm{-x+n}{n}}{\binm{x+n}{n}}\\
&&\xxqdn\:\:\times\:\bigg\{H_{n}^{\langle2\rangle}(x)-H_{n}^{\langle2\rangle}(-x)
+2\Big[H_n(x)+H_{n}(-x)\Big]\Big[H_n(x)+H_{n}(-x)-\tfrac{2}{1-x}\Big]\\
&&\xxqdn\:\:-\:\Big[H_n(x)-H_{n}(-x)\Big]\Big[H_n(x)-H_{n}(-x)-2H_{n+1}(\tfrac{x-n-2}{2})-\tfrac{2(x^2-n-n^2)}{xn(n+1)}\Big]\\
&&\xxqdn\:\:-\:\,H_{n}(\tfrac{x-n}{2})\Big[H_{n+1}(\tfrac{x-n-2}{2})+\tfrac{2(x^3-nx^2+n^2+n^3)}{x(x-n)n(n+1)}\Big]
+\tfrac{4}{(1-x)^2}-\tfrac{4(x^2-nx+n+n^2)}{x(x-n)^2(n+1)}\bigg\}.
 \enm
\end{thm}

\begin{proof}
Set $y=x$ in \eqref{whipple-e} to obtain
 \bmn\label{whipple-i}
\sum_{k=0}^n(-1)^kk\binm{n}{k}\frac{\binm{n+k}{k}}{\binm{x+k}{k}}
=(-1)^n\frac{n(n+1)}{1-x}\frac{\binm{-x+n}{n}}{\binm{x+n}{n}}.
 \emn
 Applying the derivative operator $\mathcal{D}_y^2$ to both sides of
 it, we get
\bnm
&&\xqdn\sum_{k=0}^n(-1)^k\binm{n}{k}\frac{\binm{n+k}{k}}{\binm{x+k}{k}}
 k\Big\{H_{k}^{2}(x)+H_{k}^{\langle2\rangle}(x)\Big\}\\
&&\xqdn\:=\:(-1)^n\frac{n(n+1)}{1-x}\frac{\binm{-x+n}{n}}{\binm{x+n}{n}}
 \Big\{\big[H_{n}(x)+H_{n-1}(1-x)\big]^2+H_{n}^{\langle2\rangle}(x)-H_{n-1}^{\langle2\rangle}(1-x)\Big\}.
 \enm
The difference of the last equation and Theorem \ref{thm-b} offers
Theorem \ref{thm-e}.
\end{proof}

When $x\to p$, where $p$ a nonnegative integer, we can deduce the
following four corollaries from Theorem \ref{thm-e}.

\begin{corl}[Harmonic number identity] \label{corl-r}
\bnm
 &&\sum_{k=0}^n(-1)^k\binm{n}{k}\binm{n+k}{k}kH_{k}^2=(-1)^nn(n+1)\\
 &&\:\times\:\begin{cases}
 4H_{n}^2-4H_n-2H_{n}^{\langle2\rangle}+H_{\frac{n}{2}}^{\langle2\rangle}+2,&n=0\,(\qqdn\mod2);\\[2mm]
  4H_{n}^2-4H_n-2H_{n+1}^{\langle2\rangle}+H_{\frac{n+1}{2}}^{\langle2\rangle}+\tfrac{2(n^3+2n^2+n+1)}{n(n+1)^2},&n=1\,(\qqdn\mod2).
\end{cases}
 \enm
\end{corl}

\begin{corl} \label{corl-s}
Let $p$ be a positive integer satisfying $0<p\leq n$. Then
 \bnm
 &&\sum_{k=0}^n(-1)^k\binm{n+k}{k}\binm{p+n}{n-k}kH_{p+k}^2=\frac{n(n+1)}{p(p-1)}\frac{(-1)^{n-p}}{\binm{n}{p}}\\
 &&\:\times\:\begin{cases}
 H_{n+p}+3H_{n-p}-2H_{p-2}+H_{\frac{n+p}{2}}-H_{\frac{n-p}{2}}+\frac{p}{n(n+1)},&n-p=0\,(\qqdn\mod2);\\[2mm]
  H_{n+p}+3H_{n-p}-2H_{p-2}+H_{\frac{n+p-1}{2}}-H_{\frac{n-p-1}{2}}-\frac{p}{n(n+1)},&n-p=1\,(\qqdn\mod2).
\end{cases}
 \enm
\end{corl}

\begin{corl}[Harmonic number identity] \label{corl-t}
 \bnm
 \qquad\sum_{k=0}^n(-1)^k\binm{n}{k}kH_{n+k}^2=\frac{n+1}{n-1}\frac{1}{\binm{2n}{n}}\Big\{H_{2n}-H_{n}+\tfrac{5n^2+n-2}{n^3-n}\Big\}.
 \enm
\end{corl}

\begin{corl} \label{corl-u}
Let $p$ be a positive integer with $p>n$. Then
 \bnm
 &&\xxqdn\sum_{k=0}^n(-1)^k\binm{n+k}{k}\binm{p+n}{n-k}kH_{p+k}^2=\frac{n+n^2}{2(1-p)}\binm{p-1}{n}\\
&&\xxqdn\:\times\:\bigg\{\tfrac{2}{p^2(1-p)^2}\Big[1-2p+p(1-p)(H_{p+n}+H_{p-n-1}-H_p)\Big]^2\\
&&\xxqdn\:+\:H_{p+n}^{\langle2\rangle}+H_{p-n-1}^{\langle2\rangle}-2H_{p-2}^{\langle2\rangle}-B(p,n)\bigg\},
 \enm
where the symbol $B(p,n)$ on the right hand side can be seen in
Corollary \ref{corl-h}.
\end{corl}

\begin{thm} \label{thm-f}
Let $x$ be a complex number. Then
 \bnm
&&\xqdn\sum_{k=0}^n(-1)^k\binm{n}{k}\frac{\binm{n+k}{k}}{\binm{x+k}{k}}k^2H_{k}^{2}(x)
=(-1)^n\frac{(n+n^2)(n+n^2-x)}{2(1-x)(2-x)}\frac{\binm{-x+n}{n}}{\binm{x+n}{n}}\\
&&\xqdn\:\:\times\:\bigg\{H_{n}^{\langle2\rangle}(x)\!-\!H_{n}^{\langle2\rangle}(-x)
+2\Big[H_n(x)+H_{n-2}(2-x)+\tfrac{2}{n+n^2-x}\Big]
\\&&\xqdn\:\:\times\:\Big[H_n(x)+H_{n-2}(2-x)\Big]-\Big[H_n(x)-H_{n}(-x)-2H_{n+1}(\tfrac{x-n-2}{2})+U_n(x)\Big]\\
&&\xqdn\:\:\times\:\Big[H_n(x)-H_{n}(-x)\Big]-H_{n}(\tfrac{x-n}{2})\Big[H_{n}(\tfrac{x-n}{2})-V_n(x)\Big]
+\tfrac{2(2x^2-6x+5)}{(1-x)^2(2-x)^2}-W_n(x)\bigg\},
 \enm
where the expressions $Un(x)$, $V_n(x)$ and $W_n(x)$ on the right
hand side have appeared in Theorem \ref{thm-c}.
\end{thm}

\begin{proof}
Take $y=x$ in \eqref{whipple-g} to achieve
 \bmn\label{whipple-h}
\sum_{k=0}^n(-1)^kk^2\binm{n}{k}\frac{\binm{n+k}{k}}{\binm{x+k}{k}}
=(-1)^n\frac{(n+n^2)(n+n^2-x)}{2-3x+x^2}\frac{\binm{-x+n}{n}}{\binm{x+n}{n}}.
 \emn
 Applying the derivative operator $\mathcal{D}_y^2$ to both sides of
 it, we attain
\bnm
&&\xqdn\sum_{k=0}^n(-1)^k\binm{n}{k}\frac{\binm{n+k}{k}}{\binm{x+k}{k}}
 k^2\Big\{H_{k}^{2}(x)+H_{k}^{\langle2\rangle}(x)\Big\}
 =(-1)^n\frac{(n+n^2)(n+n^2-x)}{2-3x+x^2}\frac{\binm{-x+n}{n}}{\binm{x+n}{n}}\\
&&\xqdn\:\:\times\:
 \Big\{\big[H_{n}(x)+H_{n-2}(2-x)\big]\big[H_{n}(x)+H_{n-2}(2-x)+\tfrac{2}{n+n^2-x}\big]\\
&&\quad+\:
 H_{n}^{\langle2\rangle}(x)-H_{n-2}^{\langle2\rangle}(2-x)\Big\}.
 \enm
The difference of the last equation and Theorem \ref{thm-c} produces
Theorem \ref{thm-e}.
\end{proof}

When $x\to p$, where $p$ a nonnegative integer, we can derive the
following four corollaries from Theorem \ref{thm-f}.

\begin{corl}[Harmonic number identity] \label{corl-v}
\bnm
 &&\xqdn\sum_{k=0}^n(-1)^k\binm{n}{k}\binm{n+k}{k}k^2H_{k}^2=(-1)^n\frac{n^2(n+1)^2}{2}\\
 &&\xqdn\:\times\:\begin{cases}
 4H_{n}^2-\tfrac{6n^2+6n-4}{n(n+1)}H_n-2H_{n}^{\langle2\rangle}
 +H_{\frac{n}{2}}^{\langle2\rangle}+\tfrac{7n^2+7n-4}{2n(n+1)},&n=0\,(\qqdn\mod2);\\[2mm]
  4H_{n}^2-\tfrac{6n^2+6n-4}{n(n+1)}H_n-2H_{n}^{\langle2\rangle}
 +H_{\frac{n-1}{2}}^{\langle2\rangle}+\tfrac{7n^2(n+1)^2-4}{2n^2(n+1)^2},&n=1\,(\qqdn\mod2).
\end{cases}
 \enm
\end{corl}

\begin{corl} \label{corl-w}
Let $p$ be a positive integer satisfying $0<p\leq n$. Then
 \bnm
 &&\xxqdn\qqdn\sum_{k=0}^n(-1)^k\binm{n+k}{k}\binm{p+n}{n-k}k^2H_{p+k}^2=\frac{(n+n^2)(n+n^2-p)}{p(p-1)(p-2)}\\
 &&\xxqdn\qqdn\:\:\times\:\frac{(-1)^{n-p+1}}{\binm{n}{p}}\Big\{ H_{n+p}+3H_{n-p}-2H_{p-3}+E(p,n)\Big\}.
 \enm
where the symbol on the right hand side stands for
 \bnm
 E(p,n)=\begin{cases}
H_{\frac{n+p}{2}}-H_{\frac{n-p}{2}}-\frac{p(1+p-p^2)-(2+p)(n+n^2)}{(n+n^2)(n+n^2-p)},&p-n=0\,(\qqdn\mod2);\\[2mm]
  H_{\frac{n+p-1}{2}}-H_{\frac{n-p-1}{2}}+\frac{p(1+p-p^2)+(2+p)(n+n^2)}{(n+n^2)(n+n^2-p)},&p-n=1\,(\qqdn\mod2).
\end{cases}
 \enm
\end{corl}

\begin{corl}[Harmonic number identity] \label{corl-x}
 \bnm
 \quad\sum_{k=0}^n(-1)^k\binm{n}{k}k^2H_{n+k}^2=\frac{n^2(n+1)}{(n-1)(n-2)}
 \frac{1}{\binm{2n}{n}}\Big\{H_{n}-H_{2n}-\tfrac{3n+1}{n^2}-\tfrac{5n^2-5n-4}{n^3-2n^2-n+2}\Big\}.
 \enm
\end{corl}

\begin{corl} \label{corl-y}
Let $p$ be a positive integer with $p>n$. Then
 \bnm
&&\sum_{k=0}^n(-1)^k\binm{n+k}{k}\binm{p+n}{n-k}k^2H_{p+k}^{2}=\frac{(n+n^2)(n+n^2-p)}{2(p-1)(p-2)}\binm{p-1}{n}\\
 &&\:\times\:\Big\{H_{p+n}^{\langle2\rangle}+H_{p-n}^{\langle2\rangle}-2H_{p-3}^{\langle2\rangle}
 +2\Big(H_{p+n}+H_{p-n-1}-H_{p-3}+\tfrac{2}{n+n^2-p}\Big)\\
 &&\:\times\:(H_{p+n}+H_{p-n-1}-H_{p-3})-C(p,n)\big[C(p,n)+D(p,n)\big]\\
 &&\:-\:\tfrac{2[1-2n(1+n)+(2+n)p-p^2]}{(1+n)(n+n^2-p)(n-p)}-\tfrac{2}{(p-n)^2}\Big\},
 \enm
where the expressions $C(p,n)$ and $D(p,n)$ on the right hand side
can be seen in Corollary \ref{corl-l}.
\end{corl}

 \textbf{Acknowledgments}

 The work is supported by the National Natural Science Foundations of China (Nos. 11661032, 11301120).



\end{document}